\documentclass[a4paper,11pt]{amsart}
\usepackage[utf8]{inputenc}
\usepackage{amsfonts,amssymb,amsmath,amscd,amsthm,amsbsy,hyperref,tikz}

\usepackage{setspace}
\usepackage{array}
\usepackage{setspace}
\usepackage[margin= 1.8cm]{geometry}
\usepackage{letltxmacro}
\numberwithin{equation}{section}
\usepackage{longtable}

\title{Prescribing Ricci curvature on a Product of Spheres}
\author{Timothy Buttsworth, Anusha M. Krishnan}
\address{\begin{tabular}{l}
School of Mathematics and Physics, The University of Queensland, St Lucia, Brisbane,\\ QLD 4072, Australia\\
\emph{E-mail address}: {\tt t.buttsworth@uq.edu.au}\\[0.2cm]
215 Carnegie Building, Department of Mathematics, Syracuse University, Syracuse,\\ NY 13244, USA \\
\emph{E-mail address}: {\tt akrish03@syr.edu}
\end{tabular}
}

\newtheorem{mainthm}{\sc Theorem}

\newtheorem{theorem}{Theorem}
\newtheorem{lemma}[theorem]{Lemma}

\newtheorem{prop}[theorem]{Proposition}
\newtheorem{cor}[theorem]{Corollary}

\theoremstyle{definition}

\theoremstyle{remark}
\newtheorem{remark}[theorem]{Remark}
\newtheorem*{mainrmk}{\sc Remark}

\numberwithin{equation}{section}

\newcommand{\Z}{\mathbb{Z}}
\newcommand{\R}{\mathbb{R}}

\newcommand{\G}{\mathsf{G}}
\newcommand{\K}{\mathsf{K}}
\renewcommand{\H}{\mathsf{H}}

\newcommand{\SO}{\mathsf{SO}}

\newcommand{\g}{\mathrm{g}}
\newcommand{\dd}{\mathrm{d}}
\newcommand{\Ric}{\operatorname{Ric}}

\newcommand{\h}{\mathfrak{h}}
\renewcommand{\k}{\mathfrak{k}}
\renewcommand{\gg}{\mathfrak{g}}
\newcommand{\m}{\mathfrak{m}}
\newcommand{\p}{\mathfrak{p}}
\newcommand{\n}{\mathfrak{n}}
\newcommand{\Ad}{\operatorname{Ad}}
\newcommand{\ddt}{\frac{\partial}{\partial t}}
\newcommand{\ddx}{\frac{\partial }{\partial x}}
\newcommand{\ddy}{\frac{\partial }{\partial y}}



\begin{document}
\maketitle
\begin{abstract}
 We prove an existence result for the prescribed Ricci curvature equation for
 certain doubly warped product metrics on $\mathbb{S}^{d_1+1}\times \mathbb{S}^{d_2}$, where $d_i \geq 2$. If $T$ is a metric satisfying certain curvature assumptions, we show that $T$ can be scaled independently on the two factors so as to itself be the Ricci tensor of some metric.
\end{abstract}
\section{Introduction}
The prescribed Ricci curvature problem on a manifold $M$ is the problem of finding a 
Riemannian metric $\g$ that solves the following equation for a fixed symmetric $(0,2)$-tensor field $T$ on $M$:
\begin{equation}\label{PRCE1}
\Ric(\g)=T.
\end{equation}
The Ricci tensor $\Ric(\g)$ can be expressed as the product of the first and second derivatives of $\g$, as well as the 
inverse components of $\g$ itself, so \eqref{PRCE1} is a second order nonlinear PDE in the metric, and its local and global solvability is a fundamental problem in Riemannian geometry. 
In the setting of compact K\"ahler manifolds, an important instance of this problem is the Calabi conjecture, solved by Yau in \cite{YauRicci}.

Local existence theory was studied by DeTurck in \cite{Deturck81}, where he proved that \eqref{PRCE1} 
can always be solved in a small neighbourhood of a point $p$ at which $T$ is non-degenerate. 
On the other hand, global solvability is a much more subtle problem.
In \cite{DeturckKoiso,Hamilton84}, the authors showed independently 
that for each positive definite $T$ on a compact manifold there is a constant $c_0(T)$ so that if $c > c_0$, 
$cT$ is not the Ricci tensor of any metric on $M$. Thus a more reasonable problem is finding both a constant $c>0$ and a metric $\g$ so that 
\begin{equation}\label{PRCE1'}
\Ric(\g)=cT.
\end{equation}

In \cite{Delay01}, \cite{Delay02}, and \cite{DeTurck85},
the inverse function theorem was used to solve the equation in a neighbourhood of certain Einstein metrics. 
Apart from these perturbation-type results, the only other setting where global existence has been obtained is under the assumption of symmetry, i.e., when $\g$ and $T$ are invariant under a Lie group acting on $M$. For example, if $M$ is a homogeneous space, 
and $\g$ and $T$ are assumed to be invariant under the transitive action of a group $\G$, \eqref{PRCE1'} 
becomes a system of algebraic equations. Existence results were obtained in, for example, \cite{Buttsworth19,BRPW2020,AP16,AP19}.

The next natural setting is one where a compact Lie group $\G$ acts on $M$ with cohomogeneity one, i.e., 
the generic orbits of the group action have co-dimension one. 
The simplest example is that of rotationally symmetric metrics on $\R^n$ or $\mathbb{S}^n$, where $\G = \SO(n)$. 
By symmetry, any invariant cohomogeneity one metric is parametrized by functions of a single spatial variable, 
so \eqref{PRCE1'} reduces to a system of ODEs. Related problems such as the Einstein equation ($\Ric(\g) = \lambda\g$) 
and the gradient Ricci soliton equation ($\Ric(\g) + \mathrm{Hess}(f) = \lambda\g$) have been widely studied 
in the cohomogeneity one setting by several authors (see, for example, \cite{BohmSpheres,Buttsworth18,Buttsworth19',DancerWang,EscWang}).

By contrast, the prescribed Ricci curvature problem on cohomogeneity one manifolds has remained relatively less explored. Existence results for rotationally symmetric $T$ on $\R^n$ were obtained by DeTurck and Cao in \cite{CaoDeTurck}. 
In \cite{Hamilton84}, Hamilton proved that if $T$ is itself a metric of positive Ricci curvature on $\mathbb{S}^3$ equivariant 
under $\SO(3)\times \Z_2$ (i.e. rotationally symmetric with an additional reflection symmetry about the equatorial $\mathbb{S}^2$) 
then there exists a unique $c>0$ and a metric $\g$ satisfying \eqref{PRCE1'}. 
In this article we obtain an analogous result for $\SO(d_1+1)\times \SO(d_2+1)\times \Z_2$-invariant tensors $T$ and metrics $\g$ on $\mathbb{S}^{d_1 + 1}\times \mathbb{S}^{d_2}$ (where $d_i\geq 2$) that are of the form
\begin{align*}
 T(t) &= \dd t^2 + T_1(t)\,\Omega_1^2 + T_2(t)\, \Omega_2^2, \\
 \g(t) &= h(t)^2\,\dd t^2 + f_1(t)^2\,\Omega_1^2 + f_2(t)^2 \,\Omega_2^2, \hspace{0.5cm} t\in [0,1].
\end{align*}

The action is essentially a product of the 
``rotation'' action of $\SO(d_1+1)$ on $\mathbb{S}^{d_1 + 1}$, 
with the transitive action of $\SO(d_2+1)$ on $\mathbb{S}^{d_2}$, and the $\Z_2$ factor corresponds to 
reflection symmetry about the midpoint $t=\frac{1}{2}$ of the orbit space $[0,1]$. 
The singular orbit at
$t=1$ ($t=0$) is the product of the second sphere $\mathbb{S}^{d_2}$ 
with the north (south) pole of the first sphere $\mathbb{S}^{d_1+1}$. 
Here, the quantity $T_t = T_1(t)\,\Omega_1^2 + T_2(t)\,\Omega_2^2$ is a homogeneous product metric on the generic orbit $\mathbb{S}^{d_1} \times \mathbb{S}^{d_2}$, where $\Omega_i^2$ denotes the standard metric on $\mathbb{S}^{d_i}$. Since the metric restricted to each orbit is a product, it is reasonable to expect that we will require two scaling factors $c_1$ and $c_2$ corresponding to the two factors in the product $\mathbb{S}^{d_1 + 1}\times \mathbb{S}^{d_2}$. We prove the following existence result.

\begin{mainthm}\label{RicciExistence}
Let $T = \dd t^2 + T_1(t)\,\Omega_1^2 + T_2(t)\, \Omega_2^2$ be an $\SO(d_1+1)\times \SO(d_2+1)\times \Z_2$-invariant metric on $\mathbb{S}^{d_1 + 1} \times \mathbb{S}^{d_2}$ such that $T_i'(t)\le 0$ for $t\in[\frac{1}{2},1]$, $T_1'(t)<0$ for $t\in (\frac{1}{2},1)$, and $T_1''(\frac{1}{2})<0$. Then there exist constants $c_1, c_2 >0$ and a smooth 
invariant metric $\g$ on $M$ satisfying the following prescribed Ricci curvature equation with scaling: $\Ric(\g) = c_1 (\dd t^2 + T_1\,\Omega_1^2) + c_2T_2\,\Omega_2^2$.
\end{mainthm}

\begin{mainrmk}
 Since the Ricci tensor is scaling-invariant, the set of metrics satisfying the conclusion of Theorem \ref{RicciExistence} will be 
 at least one-dimensional. In fact, since we are allowing two scaling constants $c_1$ and $c_2$, one may expect that there is a two-parameter family of solutions. This is indeed the case, because our construction involves specifying freely the positive numbers $f_1(\frac{1}{2})$ and $\sup_{t\in [\frac{1}{2},1]}\left(\frac{1+(f_2'(t))^2}{f_2(t)^2}\right)$. Although it would appear that a more natural alternative would be to specify $f_1(\frac{1}{2})$ and $f_2(\frac{1}{2})$ freely, at this stage it is unclear if this is possible. Indeed, our existence proof requires \textit{a priori} control of $\frac{1+(f_2'(t))^2}{f_2(t)^2}$ close to the singular orbit at $t=1$, and this seems difficult to achieve by merely specifying $f_2(\frac{1}{2})$.
\end{mainrmk}

To our knowledge, after Hamilton's 1984 result in \cite{Hamilton84}, this is the first global existence result for the prescribed Ricci curvature problem on closed cohomogeneity one manifolds. The assumptions on the tensor $T$ are satisfied, for example, if $T$ is a metric whose sectional curvatures satisfy $\sec(X_1, \ddt)>0$ and $\sec(X_1, X_2) \le 0$ everywhere, whenever $X_i$ is tangent to $\mathbb{S}^{d_i}$, and $\ddt$ is a vector orthogonal to the orbits. As a result, when $T_2$ is a constant, our theorem is very similar to Hamilton's result; in fact, in that situation the proof essentially reduces to that in \cite{Hamilton84} (see Section \ref{T2CC}).

One of the main challenges in proving Theorem \ref{RicciExistence} is that the solution needs to satisfy certain 
boundary conditions at $t=1$, in order to define a smooth metric on the 
compact manifold $\mathbb{S}^{d_1 + 1} \times \mathbb{S}^{d_2}$. 
For example, for $\g$ to define a metric on $\mathbb{S}^{d_1 + 1} \times \mathbb{S}^{d_2}$, we must have $f_1(1) = 0$ since $t=1$ corresponds to a fixed point of the $\SO(d_1 + 1)$ 
action on the $\mathbb{S}^{d_1 +1}$ factor. There are additional conditions on the higher derivatives of the components of $\g$,
collectively referred to as ``smoothness conditions''. 
Our proof follows Hamilton's strategy of obtaining ODE solutions starting in the middle ($t = \frac{1}{2}$) 
and showing that for some value of the scaling constant $c_1$, 
the solution extends up to the singular orbit ($t=1$) while also satisfying the smoothness conditions. 
More precisely, we obtain a sequence of solutions to an equivalent ODE system, that degenerate in the limit. 
The monotonicity assumptions on $T_1$ and $T_2$ 
ensure that the ODE solutions exist on $[\frac{1}{2},1)$, that the degeneration only occurs at the $t=1$ singular orbit, and that 
this degeneration does indeed produce the required smoothness conditions. 
The monotonicity of $T_1$ and $T_2$ is essential to our construction, and in fact, 
Hamilton produces an example of non-existence on $\mathbb{S}^3$ with $T_1$ not monotone.

The degeneration in the sequence described above corresponds to collapse ($f_1(1) = 0$). Our next result concerns regularity of solutions to the prescribed Ricci curvature equation close to the singular orbit. It essentially tells us that once we have shown that collapse occurs at the singular orbit ($t=1$), the remaining smoothness conditions will be satisfied as well. In the following theorem, the principal part of $M$ refers to the open subset of $M$ given by the union of all the generic orbits, and corresponds to $0<t<1$.
\begin{mainthm}\label{RicciRegularity}
 Let $M = \mathbb{S}^{d_1+1}\times \mathbb{S}^{d_2}$. 
 Let $\g = h(t)^2\,\dd t^2 + f_1(t)^2\,\Omega_1^2 + f_2(t)^2\,\Omega_2^2$ 
 be a smooth Riemannian metric on the principal part of $M$ whose Ricci curvature 
 $T = \dd t^2 + T_1(t)\,\Omega_1^2 + T_2(t)\,\Omega_2^2$ is a Riemannian metric which can be extended smoothly to the singular orbit at $t=1$. 
 If $f_1',f_2'\le 0$ and $\frac{f_2'}{f_2}$ is bounded close to $t=1$, and 
 \begin{equation}\label{WeakRegularity}
  \lim_{t\to 1}f_1(t)=0,
 \end{equation}
 then $\g$ can be extended smoothly to the singular  orbit at $t=1$. 
 \end{mainthm}
\begin{mainrmk}
 It turns out that the condition \eqref{WeakRegularity} in Theorem \ref{RicciRegularity} can actually be replaced by the weaker, but more complicated condition that 
 \begin{equation}\label{StrongRegularity}
  (h,\ln(f_1),\frac{f_1'}{f_1}) \ \text{is unbounded as} \ t\to 1. 
 \end{equation}
 We require this stronger form of Theorem \ref{RicciRegularity} to prove Theorem \ref{RicciExistence} because \textit{a priori}, 
 the candidate solution we eventually construct is not known to satisfy \eqref{WeakRegularity}, but it \textit{is} known to satisfy \eqref{StrongRegularity}. 
\end{mainrmk}

Regularity at the singular orbit is a challenging issue in the study of all geometric equations on closed cohomogeneity one manifolds, and 
Theorem \ref{RicciRegularity} enables us to tackle this. 
It is worth noting that most known results on the Einstein and Ricci soliton equations in the cohomogeneity one setting are on non-compact manifolds.

The paper is structured as follows. 
In Section 2, we 
provide necessary background
and use the second contracted Bianchi identity to re-write the prescribed Ricci curvature equation. 
In Section 3, we prove the regularity result of Theorem \ref{RicciRegularity} (with \eqref{WeakRegularity} replaced by \eqref{StrongRegularity}). 
In Section 4, we prove the existence result in the special case that $T_2$ is constant. 
In this case, the ODEs essentially reduce to those treated in \cite{Hamilton84}, 
and the techniques give a good warm-up for the material in the subsequent sections. 
In Section 5, we use Leray-Schauder degree theory to construct a sequence of metrics which solve the prescribed Ricci curvature problem, but do not close up smoothly at the singular orbits. 
In Section 6, we establish bounds on this sequence of solutions, 
to the extent that we can find a subsequence which converges to a solution satisfying the hypothesis of Theorem \ref{RicciRegularity}
(with \eqref{WeakRegularity} replaced by \eqref{StrongRegularity}),
thus completing the proof of Theorem \ref{RicciExistence}.

\section*{Acknowledgements}
The authors are grateful to Ramiro Lafuente, Artem Pulemotov and Wolfgang Ziller for comments on an early version of this paper. 
\section{Preliminaries}\label{Prelim}
In this section, we first provide general background on cohomogeneity one manifolds and their invariant metrics 
(for more details about the general theory of Lie group actions on manifolds and cohomogeneity one manifolds in particular, 
the reader may refer to \cite{AlexandrinoBettiol, Ziller09}). We then discuss the smoothness conditions required for an invariant metric defined on the principal part of the manifold to extend to a $C^1$ metric on the whole manifold. The reason it is sufficient to know when the metric is merely $C^1$ instead of $C^{\infty}$, will be made clear in Section \ref{PRR}. Finally, we write down the ODEs corresponding to the prescribed Ricci curvature equation with scaling in the cohomogeneity one setting, and use the second Bianchi identity to re-write the equations in a more convenient form.

\subsection{Cohomogeneity one manifolds}
A Lie group $\G$ is said to act on a manifold $M$ with cohomogeneity one if the orbit space $M/\G$ is one-dimensional, equivalently if the generic orbits of the group action have codimension one in $M$. Thus there are the following possibilities for the orbit space $M/\G$: $\mathbb{S}^1$, $\R$, a half-open interval, or a closed interval. For the manifolds considered in this article, $M/\G = [0,1]$.

If we let $\pi : M\rightarrow M/\G$ be the quotient map, then the preimages $\pi^{-1}(t)\in M$, for $t\in [0,1]$, are orbits of the group action. The preimages $\pi^{-1}(t)$ for $t\in(0,1)$ are called the principal orbits, while $B_- = \pi^{-1}(0)$ and $B_+ = \pi^{-1}(1)$ are called singular orbits. The union of all principal orbits is an open set $M_0\subset M$, and is called the principal part of $M$.

Pick a point $x_-\in B_-$ and let $\gamma(t)$ be a minimal $T$-geodesic normal to $B_-$, and with $\gamma(0) = x_-$, that meets $B_+$ for the first time at $\gamma(1) = x_+$. Then $\gamma(t)$ for $t \in [0,1]$ parametrizes the orbit space. 
The isotropy group at all points $\gamma(t)$, $t\in (0,1)$ is the same subgroup $\H\subset \G$, and is called the principal isotropy group. The isotropy groups $\K_\pm$ at $x_\pm$ are called the singular isotropy groups.

By the slice theorem, $M$ can be constructed as the union of two disk bundles over the singular orbits $\G/\K_\pm$, glued along their common boundary $\G/\H$. 
As a result, $\K_\pm/\H$ must be spheres $\mathbb{S}^{\ell_\pm}$. 
The data $\H\subset \K_\pm \subset \G$ is called a group diagram, and it determines the manifold up to an equivariant diffeomorphism.

By symmetry, any $\G$-invariant metric on $M$ is determined by its value at points $\gamma(t)$. We write
\begin{align*}
 \g = h(t)^2 \dd t^2 + \g_t, \hspace{0.4cm} t\in(0, 1),
\end{align*}
where $\g_t$ is a homogeneous metric on $\G/\H$.

Fix a bi-invariant metric $Q$ on $\mathfrak{g}$. The tangent space of $\G/\H$ at the point $\H$ can be identified with $\n = \h^\perp$. Then $\H$ acts on $\n$ by the adjoint action, and under this action we can write $\n$ as a sum of irreducible $\H$-modules:
\begin{align*}
 \n = \n_1 \oplus \cdots \oplus \n_s.
\end{align*}

\subsection{The $2$-summand case} \label{subsec:2summand}
In this article we make the assumption that under the isotropy action of $\H$, $\n$ splits as a sum of two irreducible modules, $\h^\perp = \n_1 \oplus \n_2$. We will further assume that $d_i = \dim \n_i > 1$.
We will focus on diagonal metrics, i.e. metrics for which $\n_1$ and $\n_2$ are orthogonal:
\begin{align*}
 \g_t (X,Y) = f_1(t)^2 Q(pr_{\n_1}X, pr_{\n_1}Y) + f_2(t)^2 Q(pr_{\n_2}X, pr_{\n_2}Y), \hspace{0.4cm} X, Y \in \n.
\end{align*}
For convenience of notation we may also write this as 
\begin{align}\label{eqn:metric2summand}
 \g(t) = h(t)^2 \dd t^2 + f_1(t)^2 \Omega_1^2 + f_2(t)^2 \Omega_2^2,
\end{align}
where $\Omega_i^2$ is a symmetric bilinear form on $T_{\H}\G/\H$ defined by $\Omega_i^2(X, Y) = Q(pr_{\n_i}X, pr_{\n_i}Y)$. In general, a diagonal $\G$-invariant symmetric $2$-tensor $T$ on $M$ can be written as
\begin{align}\label{eqn:T2summand}
 T(t) = T_0(t) \dd t^2 + T_1(t) \Omega_1^2 + T_2(t) \Omega_2^2.
\end{align}
If the $\H$-modules $\n_i$ are inequivalent then any invariant symmetric $2$-tensor is diagonal.
 
For the existence result, we assume that the manifold $M$ is a product 
$M \cong \mathbb{S}^{d_1 +1} \times \G_2/\K_2$ where $\mathbb{S}^{d_1} = \G_1/\K_1$ 
is an isotropy irreducible sphere and $\G_2/\K_2$ is an isotropy irreducible space of dimension $d_2$. 
Note that when the principal orbit is a product $\G_1/\K_1 \times \G_2/\K_2$, then the modules $\n_1$ and $\n_2$ are inequivalent, 
and so the diagonal form is automatic. 
In this case, the two singular isotropy groups are the same, that is, $\K_\pm = \K$, and the group diagram satisfies 
$\G = \G_1\times \G_2$, $\K = \G_1 \times \K_2$, $\H = \K_1 \times \K_2$. 
The product of spheres addressed in Theorem \ref{RicciExistence} 
is a special case of this, where $\G_2/\K_2 = \SO(d_2 + 1)/ \SO(d_2) \cong \mathbb{S}^{d_2}$.

\subsection{Smoothness conditions}
The expressions \eqref{eqn:metric2summand} and \eqref{eqn:T2summand} define the metric (or in general, the symmetric $2$-tensor) on the principal part $M_0$ of the manifold. 
In order for it to extend smoothly to a Riemannian metric on all of $M$, 
its components must satisfy some conditions at the end-points $t=0$ and $t=1$ of the orbit space; these are referred to as \emph{smoothness conditions}. 
In \cite{VZ18}, Verdiani and Ziller provide a detailed discussion of these conditions. 
Here, we use results from that paper to derive the conditions that guarantee a $C^1$ metric.

Let us examine the metric at a particular singular orbit,
where the group diagram is $\H \subset \K \subset \G$. The normal bundle in $M$ of the singular orbit $\G/ \K$ is a 
vector bundle of the form $\G\times_{\K}V$. Here $V\cong \R^{\ell + 1}$ is a real vector space of dimension $\ell + 1$ corresponding to the sphere $\mathbb{S}^\ell = \K/\H$. Let $\mathbb{D}^{\ell + 1} \subset V$ 
be a disk. We can identify $\mathbb{D}^{\ell + 1}$ with a normal disk to $\G/\K$ at a point $x_0$, called the slice of the group action at $x_0$. As before, we denote by $\gamma(t)$ a geodesic normal to the singular orbit $\G/\K$ with $\gamma(0) = x_0$. By the slice theorem, the metric is smooth if its restriction to the slice is smooth.

We have the following diagram of Lie algebras: $\h \subset \k \subset \gg$. Let $\p$ be an $\Ad_\H$-invariant complement of $\h \subset \k$ and let $\m$ be an $\Ad_\K$-invariant complement of $\k \subset \gg$. Then we can identify $V \cong \R\cdot \gamma'(0) \oplus \p$. The tangent space to the principal orbit is identified with $\h^\perp = \p \oplus \m$.

With this notation in hand, we can prove sufficient conditions for the metric \eqref{eqn:metric2summand} to be $C^1$.
\begin{theorem}\label{C1RG}
 Suppose $h, f_1, f_2\in C^1[\frac{1}{2},1]$ and that $f_1(1) = 0$, $f_1'(1) = -h(1)<0$, $h'(1) = 0$, $f_1''(1) = 0$, $f_2(1)>0$ 
 and $f_2'(1) = 0$. Then the metric $\g$ given by \eqref{eqn:metric2summand} is $C^1$.
\end{theorem}
\begin{proof}
By \cite[Proposition 2.1]{VZ18}, 
we need only consider regularity of the metric when restricted to finitely many $2$-planes in the slice. 
In our case, $\p = \n_1$ is irreducible, hence we need only consider regularity of the metric when restricted to a single $2$-plane in the slice. In particular, we can pick a vector $v_1 \in \n_1$ and let $\mathbb{D} = \operatorname{span}\{\gamma'(0), v_1\}$. 
We need to characterize the regularity of the metric when restricted to the $2$-disk $\mathbb{D}^2 = \exp_{x_0}(\mathbb{D})$. 
Further, we can separately consider the inner product amongst vectors in the slice and amongst vectors in $\m$. Note that in our case, $\m = \n_2$.

Then, smoothness ($C^1$) conditions for inner products in the slice are addressed by Corollary \ref{cor:C1slice} along with a change of coordinates to account for the fact that we are considering regularity at $t=1$ instead of $t=0$. Corollary \ref{cor:C1m} yields the condition for the inner products in $\m = \n_2$ to be $C^1$.
\end{proof}

Let us now prove the results that go into the proof of Theorem \ref{C1RG}. First we discuss regularity of the metric on the slice. 
Since we can restrict to a $2$-plane in the slice, we work with rotationally symmetric metrics on the disc $\mathbb{D}^2$. We are interested in the conditions for the metric to be $C^1$ 
with respect to Cartesian coordinates $x, y$ at the origin. We have
\begin{align}
\begin{split}
\label{eqn:cart_polar}
 x &= t\cos \theta, \quad y = t\sin \theta,\\
 \dd x &= \cos \theta \dd t - t \sin \theta \dd \theta, \quad \dd y = \sin \theta \dd t + t \cos \theta \dd \theta,\\
 \implies \dd t &= \cos \theta \dd x + \sin \theta \dd y, \quad \dd \theta = -\frac{\sin \theta}{t}\dd x + \frac{\cos \theta}{t}\dd y. 
\end{split}
\end{align}
\begin{lemma}\label{lem:C1slice}
 Consider the metric $\g = \dd t^2 + f_1(t)^2\dd \theta^2$ on $\mathbb{D}^2$. 
 Suppose that $f_1 \in C^1[0,1]$ and that $f_1(0) = 0$, $f_1'(0) = 1$ and $f_1''(0)$ exists with $f_1''(0) = 0$. 
 Then $\g$ is $C^1$ on $\mathbb{D}^2$ with respect to Cartesian coordinates.
\end{lemma}
\begin{proof}
Since the coordinates $(t, \theta)$ are smoothly equivalent to the Cartesian coordinates 
$(x, y)$ at points other than the origin,
at points $(x,y) \neq (0,0)$ the metric is $C^1$ if and only if $f_1 \in C^1(0,1]$. Therefore, we are left to characterize regularity at the origin.

First, using \eqref{eqn:cart_polar} we find the components of the metric with respect to Cartesian coordinates.
\begin{align*}
 \g_{xx} = 1 + \left( \frac{f_1(t)^2}{t^2} - 1 \right)\sin^2 \theta, \quad \g_{yy} = 1 + \left( \frac{f_1(t)^2}{t^2} - 1 \right)\cos^2 \theta, \quad
 \g_{xy} = 2\left( \frac{f_1(t)^2}{t^2} - 1 \right)\cos \theta \sin \theta
\end{align*}
and the metric at the origin is the Euclidean metric satisfying $\g_{xx} = \g_{yy} = 1$, $\g_{xy} = 0$.

Since $f_1(0) = 0$ and $f_1'(0) = 1$, we find $\displaystyle\lim_{t\to 0} \frac{f_1(t)}{t} = f_1'(0) = 1$. Hence $\displaystyle\lim_{t\to 0}\left( \frac{f_1(t)^2}{t^2} - 1 \right) = 0$. As a result, 
$\displaystyle\lim_{(x,y) \to (0,0)}\g_{xx} = \displaystyle\lim_{t\to 0}\g_{xx}$ exists and equals $\g_{xx}|_{(0,0)}$ so $\g_{xx}$ is continuous at the origin. Similarly $\g_{yy}$ and $\g_{xy}$ are also continuous at the origin, so we now know
that $\g \in C^0$.

Now, using the chain rule we see that at points $(x,y)$ other than the origin,
 \begin{align*}
  \ddx\g_{yy} &=  \left( \frac{2f_1f_1'}{t^2} - \frac{2f_1^2}{t^3} \right)\cos^3\theta + 2\left( \frac{f_1^2}{t^3} - \frac{1}{t} \right)\cos\theta \sin^2\theta.
 \end{align*}
 Using $f_1(0) = 0$ and $f_1''(0) = 0$ along with L'H\^opital's rule yields
 \begin{align*}
  \lim_{t\to 0} \left( \frac{2f_1f_1'}{t^2} - \frac{2f_1^2}{t^3} \right) = 0 \quad \mbox{and} \quad \lim_{t\to 0} \left( \frac{f_1^2}{t^3} - \frac{1}{t} \right) = 0.
 \end{align*}
Hence $\displaystyle\lim_{(x,y)\to (0,0)} \ddx\g_{yy}$ exists and equals zero. On the other hand, a similar computation using L'H\^opital's rule yields that $\ddx \g_{yy}|_{(0,0)} = 0$. Thus we see that $\ddx \g_{yy}$ is continuous at the origin. Similarly we can show that the other first partial derivatives of all the metric components are continuous at the origin. 
We conclude that $\g \in C^1$.
\end{proof}
When the curve $\gamma$ is not parametrized by the arc length of $\g$, the conditions are slightly modified as follows:
\begin{cor} \label{cor:C1slice}
 Consider the metric $\g = h(t)^2\dd t^2 + f_1(t)^2\dd \theta^2$. Suppose we have $h, f_1\in C^1[0,1]$ and $f_1(0) = 0$, $h(0) \neq 0$, $f_1'(0) = h(0)$, $h'(0) = 0$, 
 $f_1''(0) = 0$. Then the metric $\g$ is $C^1$ at the origin.
\end{cor}
\begin{proof}
 Let $s$ be the arclength parameter along the geodesic $\gamma$. Then $\frac{\partial}{\partial s} = \frac{1}{h}\ddt$.  
 Also note that $s=0$ coincides with $t = 0$. 
 Therefore, $f_1'(0) = h(0)$ implies $\frac{\partial}{\partial s} f_1(0) = 1$. 
 Further, $f_1''(0) = 0$ and $h'(0)=0$ imply that $\frac{\partial^2}{\partial s^2}f_1(0) = 0$.
 
 Therefore by Lemma \ref{lem:C1slice}, the metric is $C^1$ in the Cartesian coordinates $(x_1, y_1)$ corresponding to $s, \theta$. By Lemma \ref{lem:rotsymC1} and $h'(0) = 0$ one sees that the transformation to the Cartesian coordinates $(x,y)$ corresponding to $(t, \theta)$ is $C^2$. 
 Hence we conclude that $\g$ is $C^1$ with respect to the    coordinates $(x,y)$.
\end{proof}
To study the regularity of the metric on $\m$, we can examine the restriction $\g|_{\m}$ at points on $\mathbb{D}^2$. 
Then \cite[Lemma 3.5]{VZ18} implies that $f_2(t)^2$ must extend to a rotationally symmetric function on $\mathbb{D}^2$. The next result characterizes regularity for rotationally symmetric functions on the disc.
\begin{lemma} \label{lem:rotsymC1}
 Let $G$ be a rotationally symmetric function on $\mathbb{D}^2$, that is, 
 $G(x,y) = F(t)$ for some function $F$, where $t = \sqrt{x^2 + y^2}$. Then $G$ is $C^1$ on $\mathbb{D}^2$ if and only if $F\in C^1[0,1]$ with $F'(0) = 0$.
\end{lemma}
\begin{proof}
 As before, we only need to provide conditions at the origin. 
 Suppose $F$ is a $C^1$ function of $t$ satisfying $F'(0) = 0$. We see that $\displaystyle\lim_{(x,y)\to(0,0)} G(x,y) = \lim_{t\to 0}F(t) = F(0) = G(0,0)$, so $G$ is continuous at the origin. Next, we calculate at points other than the origin:
 \begin{align*}
  &\ddx G (x,y) = \ddx F(t) = F'(t)\cos \theta\\
  \implies &\lim_{(x,y)\to (0,0)} \ddx G (x,y) = \lim_{t\to 0 }F'(t)\cos\theta = 0 \quad \mbox{since $F'$ is continuous and $F'(0) = 0$.}
 \end{align*}
 On the other hand,
 \begin{align*}
  \lim_{h\to 0^+} \frac{G(h,0) - G(0,0)}{h} = \lim_{h\to 0^+} \frac{F(h) - F(0)}{h} = F'(0) = 0.
 \end{align*}
 This equation coupled with the observation that $G(-h,0)=G(h,0)$ implies that $\ddx G(0,0)=0$, so $\ddx G(x,y)$ is continuous at the origin. Similarly, we can see that $\ddy G (x,y)$ is also continuous at the origin. Therefore $G \in C^1$.
 
 For the converse, assume that $G\in C^1$. Then the restriction of $G$ to the $x$-axis defined by $F_1(x) = G(x,0)$ is a $C^1$ function of $x$, and agrees with $F(x)$ for $x>0$. Since $G$ is rotationally symmetric, we have $G(h,0) = G(-h,0)$, so $F_1(h) = F_1(-h)$. We see that
 \begin{align*}
  & F_1'(0) = \lim_{h\to 0} \frac{F_1(h) - F_1(0)}{h} = \lim_{h\to 0} \frac{F_1(-h) - F_1(0)}{h} = -\lim_{h\to 0} \frac{F_1(-h) - F_1(0)}{-h} = -F_1'(0).
 \end{align*}
 Hence $F'(0) = 0$.
\end{proof}
Therefore, we have
 \begin{cor} \label{cor:C1m}
  Suppose $f_2\in C^1[0,1]$ and $f_2'(0) = 0$, $f_2(0)>0$. Then the metric $\g$ is $C^1$ on $\m$.
 \end{cor}
Therefore, to summarize, in order for the metric $\g$ to be $C^1$, we must ensure that:
\begin{align}\label{SC}
 f_1(1) = 0; \quad f_1'(1) = -h(1)<0; \quad h'(1) = 0; \quad f_1''(1) = 0; \quad f_2'(1) = 0; \quad f_2(1)>0.
\end{align}
Since the $2$-tensor $T$ is assumed to be smooth, $T_2''(1)$ must exist, and the $T_i$ functions must also satisfy:
\begin{align}
 T_1(1) = 0; \quad T_1'(1) = 0; \quad T_1''(1) = 2\,T_0(1)^2>0; \quad T_2'(1) = 0; \quad T_2(1)>0. 
\end{align}
Additionally, the assumption of reflection-symmetry implies that all functions must be even at $t = \frac{1}{2}$:
\begin{align}\label{EC}
 T_0'(\tfrac{1}{2}) = 0; \quad T_1'(\tfrac{1}{2}) = 0; \quad T_2'(\tfrac{1}{2}) = 0; \quad f_1'(\tfrac{1}{2}) = 0; \quad f_2'(\tfrac{1}{2}) = 0; \quad h'(\tfrac{1}{2}) = 0.
\end{align}
\subsection{The Cohomogeneity One Prescribed Ricci Curvature Equation}
Using Lemma 3.1 of \cite{Pulemotov16}, or Proposition 1.14 of \cite{GroveZiller}, we see that the prescribed Ricci curvature equation for the metric \eqref{eqn:metric2summand} is 
\begin{align}\label{PRC2S''}
\begin{split}
 d_1\frac{f_1''}{f_1}+d_2\frac{f_2''}{f_2}-\frac{h'}{h}\left(\frac{d_1f_1'}{f_1}+\frac{d_2f_2'}{f_2}\right)&=-c_1\\
 \frac{f_1''}{h^2f_1}+(d_1-1)\frac{(f_1')^2}{h^2f_1^2}+d_2\frac{f_1'f_2'}{h^2f_1f_2}-\frac{h'f_1'}{h^3 f_1}-\frac{1}{f_1^2}(d_1-1)-\frac{d_2}{d_1}\beta\frac{f_1^2}{f_2^4}&=-\frac{c_1T_1}{f_1^2}\\
 \frac{f_2''}{h^2f_2}+(d_2-1)\frac{(f_2')^2}{h^2f_2^2}+d_1\frac{f_1'f_2'}{h^2f_1f_2}-\frac{h'f_2'}{h^3 f_2}-\frac{\alpha}{f_2^2}+2\beta \frac{f_1^2}{f_2^4}&=-\frac{c_2T_2}{f_2^2},\\
\end{split}
\end{align}
where $d_i$ is the dimension of $\mathfrak{n}_i$, and the non-negative constants $\alpha$ and $\beta$ are given by
\begin{align*}
 \alpha &= \sum_{e_i \in \n_2} \|[e_1, e_i]_{\h}\|_Q^2 + \sum_{e_i \in \n_2} \|[e_1, e_i]_{\n_1}\|_Q^2 + \frac{1}{4}\sum_{e_i \in \n_2} \|[e_1, e_i]_{\n_2}\|_Q^2\\
 \beta &= \frac{1}{4}\sum_{e_i \in \n_2} \|[e_1, e_i]_{\n_1}\|_Q^2, \mbox{ where $\{e_i\}_{i = 1}^{d_2}$ is a basis for $\n_2$.}
\end{align*}
Note that $\alpha = 0$ 
would imply that all brackets in $\n_2$ are zero, equivalently, that the singular orbit $\G/\K$ is flat. We therefore assume throughout this paper that $\alpha>0$. One consequence of this is that we necessarily have $d_2>1$, which will be used implicitly throughout this paper as well. Also note that $\beta = 0$ is equivalent to $[\n_1, \n_2] = 0$. We allow $\beta\ge 0$ for Theorem \ref{RicciRegularity}, but require that $\beta=0$ for the main existence 
result of Theorem \ref{RicciExistence}. In fact, $\beta = 0$ is satisfied if $M$ is a product manifold of the form described in Section \ref{subsec:2summand}, since in that case the principal orbits themselves are products.
Finally, since we assume that 
$T = T_0(t)\, \dd t^2 + T_1(t)\, \Omega_1^2 + T_2(t)\, \Omega_2^2$ is positive-definite, then we can re-parametrise so that $T_0$ is constant. 
Since we are allowing the tensor $T$ to be scaled, we may as well assume that $T_0=1$, which is why no such term appears in \eqref{PRC2S''}.
\begin{remark}
Suppose $\beta=0$ (the main case of interest in this paper). Our reason for including two scaling constants $c_1,c_2>0$ in \eqref{PRC2S''} is that, at least in the case that $T_2>0$ is constant, 
the choice of $c_2$ is unique and $f_2$ is necessarily constant (as seen in Proposition \ref{MPF2}). This is reflective of the fact that we are solving 
the prescribed Ricci curvature equation on product manifolds. 
\end{remark}
\begin{remark}
If $\beta>0$, assuming that $T_2$ is constant no longer implies that $f_2$ is constant, but it seems as though we still need to allow $T$ to be scaled by two different scaling constants $c_1,c_2>0$. Indeed, the second equation of \eqref{PRC2S''}, coupled with $f_1(0)=f_1(1)=0$ implies that if $T_1$ is small, then $c_1$ must be large, whereas 
the third equation of \eqref{PRC2S''}, coupled with $f_2'(0)=f_2'(1)=0$, implies that if $T_2$ is large, then $c_2$ must be small. 
It therefore appears that in full generality, scaling $T$ by a single constant $c$ (i.e., insisting that $c_1=c_2$) is not enough to solve \eqref{PRC2S''}. 
\end{remark}

\subsection{The Second Contracted Bianchi Identity}
As is usual for the study of geometric equations featuring the Ricci curvature, the second contracted Bianchi 
identity plays a crucial role. In this paper, the Bianchi identity allows us to rewrite \eqref{PRC2S''} in a more convenient form. 
Indeed, Lemma 3.2 of \cite{Pulemotov16}, which is essentially the second contracted Bianchi identity, implies that whenever the 
second and third equations of \eqref{PRC2S''} hold on $[\frac{1}{2},1)$, we have
\begin{align}\label{SCBI}
\frac{\sigma'}{2h^2}&=\sigma \frac{h'}{h^3}+d_1\left(\frac{c_1T_1'}{2f_1^2}-\frac{\sigma f_1'}{h^2 f_1}\right)+d_2\left(\frac{c_2T_2'}{2f_2^2}-\frac{\sigma f_2'}{h^2 f_2}\right),
\end{align}
where $\sigma=- d_1\frac{f_1''}{f_1}-d_2\frac{f_2''}{f_2}+\frac{h'}{h}\left(\frac{d_1f_1'}{f_1}+\frac{d_2f_2'}{f_2}\right)$. 
Therefore, if in addition the first equation of \eqref{PRC2S''} holds on 
$[\frac{1}{2},1)$, 
then $\sigma=c_1$ everywhere, so
\begin{align}\label{MSCBIE}
    -\frac{h'}{h^3}=d_1\left(\frac{T_1'}{2f_1^2}-\frac{f_1'}{h^2 f_1}\right)+
    d_2\left(\frac{c_2T_2'}{2c_1f_2^2}-\frac{f_2'}{h^2 f_2}\right).
\end{align}
Therefore, if all three equations of \eqref{PRC2S''} hold, then
\begin{align}\label{NEF2S}
\begin{split}
 -\frac{h'}{h^3}&=d_1\left(\frac{T_1'}{2f_1^2}-\frac{f_1'}{h^2 f_1}\right)+
    d_2\left(\frac{c_2T_2'}{2c_1f_2^2}-\frac{f_2'}{h^2 f_2}\right),\\
\frac{f_1''}{f_1}-\frac{(f_1')^2}{f_1^2}&=h^2\left(\frac{d_1-1-c_1T_1}{f_1^2}-\frac{f_1'}{f_1}\left(\frac{d_1T_1'}{2f_1^2}+\frac{d_2c_2T_2'}{2c_1f_2^2}\right)+\beta \frac{d_2}{d_1}\frac{f_1^2}{f_2^4}\right),\\
\frac{f_2''}{f_2}-\frac{(f_2')^2}{f_2^2}&=h^2\left(\frac{\alpha-c_2T_2}{f_2^2}-\frac{f_2'}{f_2}\left(\frac{d_1T_1'}{2f_1^2}+\frac{d_2c_2T_2'}{2c_1f_2^2}\right)-2\beta \frac{f_1^2}{f_2^4}\right),\\
c_1&= -h^2\left(\frac{d_1(d_1-1-c_1T_1)}{f_1^2}+\frac{d_2(\alpha-c_2T_2)}{f_2^2}-d_2\beta \frac{f_1^2}{f_2^4}\right)\vert_{t=\frac{1}{2}}.
\end{split}
\end{align}
We make the following observation.
\begin{prop}\label{EOIE}
Let $(h,f_1,f_2)$ be a triplet of positive smooth functions on $[\frac{1}{2},1)$ that are even about $\frac{1}{2}$. The triplet solves \eqref{PRC2S''} if and only if 
it solves \eqref{NEF2S}. 
\end{prop}
\begin{proof}
We have already shown that a solution of \eqref{PRC2S''} also solves \eqref{NEF2S}. 
On the other hand, if \eqref{NEF2S} holds for smooth $(h,f_1,f_2)$, we see that the second and third equations of \eqref{PRC2S''} hold. Then again considering the quantity $\sigma=
- d_1\frac{f_1''}{f_1}-d_2\frac{f_2''}{f_2}+\frac{h'}{h}\left(\frac{d_1f_1'}{f_1}+\frac{d_2f_2'}{f_2}\right)$, \eqref{SCBI} implies that 
\begin{align*}
\frac{\sigma'}{2h^2}&=\sigma \frac{h'}{h^3}+d_1\left(\frac{c_1T_1'}{2f_1^2}-\frac{\sigma f_1'}{h^2 f_1}\right)+d_2\left(\frac{c_2T_2'}{2f_2^2}-\frac{\sigma f_2'}{h^2 f_2}\right)\\
&=-\sigma d_1\left(\frac{T_1'}{2f_1^2}-\frac{f_1'}{h^2 f_1}\right)-\sigma d_2\left(\frac{T_2'}{2f_2^2}-\frac{f_2'}{h^2 f_2}\right)+d_1\left(\frac{c_1T_1'}{2f_1^2}-\frac{\sigma f_1'}{h^2 f_1}\right)+d_2\left(\frac{c_2T_2'}{2f_2^2}-\frac{\sigma f_2'}{h^2 f_2}\right)\\
&=(c_1-\sigma)\left(\frac{d_1T_1'}{2f_1^2}+\frac{d_2c_2T_2'}{2c_1f_2^2}\right).
\end{align*}
Now the fourth equation of \eqref{NEF2S} implies that $\sigma=c_1$ at $t=\frac{1}{2}$, 
so this computation reveals that in fact $\sigma=c_1$ for all $t\in [\frac{1}{2},1)$, so the first equation of \eqref{PRC2S''} also holds. 
\end{proof}
In light of Proposition \ref{EOIE}, we can solve \eqref{NEF2S} instead of \eqref{PRC2S''}. Further, we introduce the new functions $y_i=\ln(f_i)$ and re-write \eqref{NEF2S} in terms of the $y_i$ functions as follows: 
\begin{align}\label{NEF2S'}
\begin{split}
y_1''&=h^2\left((d_1-1-c_1T_1)e^{-2y_1}-y_1'\left(\tfrac{d_1}{2}T_1'e^{-2y_1}+\tfrac{d_2c_2}{2c_1}T_2'e^{-2y_2}\right)+\beta\tfrac{d_2}{d_1}e^{2y_1-4y_2}\right),\\
y_2''&=h^2\left((\alpha-c_2T_2)e^{-2y_2}-y_2'\left(\tfrac{d_1}{2}T_1'e^{-2y_1}+\tfrac{d_2c_2}{2c_1}T_2'e^{-2y_2}\right)-2\beta e^{2y_1-4y_2}\right),\\
-\frac{h'}{h^3}&=d_1\left(\frac{T_1'}{2}e^{-2y_1}-\frac{y_1'}{h^2}\right) + d_2\left(\frac{c_2T_2'}{2c_1}e^{-2y_2} -\frac{y_2'}{h^2}\right),\\
c_1&= -h^2\left((d_1(d_1-1-c_1T_1))e^{-2y_1}+(d_2(\alpha-c_2T_2))e^{-2y_2}-d_2\beta e^{2y_1-4y_2}\right)\vert_{t=\frac{1}{2}}.
\end{split}
\end{align}
We also note that whenever \eqref{NEF2S'} (and hence \eqref{PRC2S''}) holds, we also have
\begin{align}\label{NEF2S'H2}
 h^2=\frac{tr(y')^2-tr((y')^2)-c_1}{d_1 (d_1-1-c_1T_1) e^{-2y_1} + d_2 (\alpha-c_2 T_2) e^{-2y_2} -d_2\beta e^{2y_1-4y_2}},
\end{align}
at least whenever the denominator is non-zero. Here, $tr(a,b)=d_1a+d_2b$. 
Although this auxiliary equation appears to have a singularity whenever the denominator is zero, this is not an issue we need to address in this paper. In fact, the only times that we reference this equation are in situations when it is clear that the denominator has a sign.

\section{Regularity}\label{PRR}
In this section, we prove Theorem \ref{RicciRegularity}, with \eqref{WeakRegularity} replaced by the weaker assumption \eqref{StrongRegularity}. 
The proof will follow from the following result. 
\begin{theorem}\label{LS:R}
 Suppose $(h>0,y_1,y_2)$ solves the first three equations \eqref{NEF2S'} and \eqref{NEF2S'H2} on $[1-\epsilon,1)$ for some $c_1,c_2,\epsilon>0$, and 
 $(h(t),y_1(t),y_1'(t))$ is unbounded about $t=1$. 
 Suppose also that $\left|y_2'\right|$ is bounded on $[1-\epsilon,1)$ and $y_i'\le 0$ on $[1-\epsilon,1)$. Then $h\in C^1[1-\epsilon,1]$, $f_i=e^{y_i}\in C^2[1-\epsilon,1]$, and the following is true: 
 \begin{align}\label{SCST}
 \begin{split}
  f_1(1)=0, \qquad f_2(1)>0,\\
  f_1'(1)=-h(1)<0, \qquad f_2'(1)=0,\\
  f_1''(1)=0, \qquad h'(1)=0.
  \end{split}
 \end{align}
\end{theorem}
The proof of this Theorem is achieved through a number of lemmas which prove progressively stronger statements about the regularity of the functions $h,f_1,f_2$.
\begin{lemma}\label{0OSC}
 If $(h,y_1,y_2)$ is as in the hypothesis of Theorem \ref{LS:R}, then the functions $f_i(t) = e^{y_i(t)}$ can be extended continuously to 
 $[1-\epsilon,1]$, and $f_1(1)=0$, $f_2(1)>0$. Also $\lim_{t\to 1}y_1'(t)=-\infty$. 
\end{lemma}
\begin{proof} 
Since $y_2'$ is non-positive and bounded, it is clear that $y_2$ can be extended continuously to $t=1$. Then the same is true for $f_2(t)=e^{y_2(t)}$, and $f_2(1)>0$. To verify the other claims, first 
note that by the smoothness conditions for $T_1,T_2$, the monotonicity of $y_1,y_2$ and the continuity of $e^{2y_2},e^{-2y_2}$ and $e^{2y_1}$, 
we can write the first equation of
\eqref{NEF2S'} as 
\begin{align}\label{REFY1}
 y_1''(t)=h^2(t)e^{-2y_1(t)}\left(a(t)+b(t)(1-t)y_1'(t)\right),
\end{align}
where $a$ and $b$ are both continuous functions, $a(1)>0$ and $b(1)>0$.

To show that $f_1(t)$ can be extended continuously to $t=1$ with $f_1(1)=0$, it suffices to show that $e^{-2y_1}$ is unbounded, since $y_1'\le 0$. Let us now assume that $e^{-2y_1}$ is bounded, so that $e^{-2y_1}$ is actually continuous on $[1-\epsilon,1]$. Using $y_i'\le 0$, the third equation of \eqref{NEF2S'} implies that $\frac{h'}{h^3}\le-d_1\frac{T_1'}{2}e^{-2y_1} - d_2\frac{c_2T_2'}{2c_1}e^{-2y_2}$. The continuity of $e^{-2y_1}$ and $e^{-2y_2}$ and the smoothness conditions for $T_1,T_2$ then imply that there is a constant $C'>0$ so that $h'\le C'h^3(1-t)$. Therefore, either $h$ is uniformly bounded on 
$[1-\epsilon,1)$ or $h(t)\ge \frac{1}{\sqrt{C'}(1-t)}$. In the second case, we can write $z(t)=(1-t)y_1'(t)$ so \eqref{REFY1} becomes 
\begin{align*}
 z'(t)&=-y_1'(t)+(1-t)y_1''(t)\\
 &=-\frac{z(t)}{(1-t)}+(1-t)h^2(t)e^{-2y_1(t)}\left(a(t)+b(t)z(t)\right)\\
 &\ge \frac{1}{1-t}\left(-z(t)+\frac{1}{C'}e^{-2y_1(t)}\left(a(t)+b(t)z(t)\right)\right).
\end{align*}
Since $z\le 0$ on $[1-\epsilon,1)$, this inequality tells us that $\sup_t z(t)<0$, so $y_1'(t)$ goes to $-\infty$ like $\frac{-1}{1-t}$, which contradicts the boundedness of $e^{-2y_1}$. We therefore find that $h$ is bounded on $[1-\epsilon,1)$. 
Since we now have bounds on $h$ as well as $e^{-2y_1}$, \eqref{REFY1} then implies that $y_1'$ is bounded as well. 
We therefore find bounds on $(h,y_1,y_2)$ in $C^0\times C^1\times C^1$, which is a contradiction. 
Therefore we can conclude that $\lim_{t\to 1}e^{-2y_1(t)}=\infty$.

Finally, if it is not true that $\lim_{t\to 1}y_1'(t)=-\infty$, then $y_1'(t)$ is bounded from below for values of $t$ arbitrarily close to $1$, so \eqref{REFY1} can again be used to find that $y_1'$ is bounded from below on the entirety of $[1-\epsilon,1)$, which contradicts the unboundedness of $e^{-2y_1}$. 
\end{proof}
\begin{lemma}\label{1OSC}
 If $(h,y_1,y_2)$ is as in Lemma \ref{0OSC}, then $f_1\in C^1[1-\epsilon,1]$, $h\in C^0[1-\epsilon,1]$, and $f_1'(1)=-h(1)$.
\end{lemma}
\begin{proof}
We introduce a new function $z_1$ such that $y_1'(t)=-\frac{z_1(t)}{1-t}$. Then \eqref{NEF2S'} and \eqref{NEF2S'H2} imply that
\begin{align*}
 &-\frac{z_1'}{(1-t)}-\frac{z_1}{(1-t)^2}=\\
 &\frac{\left(d_1(d_1-1)\frac{z_1^2}{(1-t)^2}-K(t)\frac{z_1}{1-t} \right)\left(d_1-1-c_1T_1-y_1'(\frac{d_1}{2}T_1' + \frac{d_2c_2}{2c_1}T_2'e^{2y_1-2y_2}) + \beta \frac{d_2}{d_1}e^{4y_1-4y_2}\right)}{d_1(d_1-1-c_1T_1)+d_2 e^{2y_1-2y_2}(\alpha-c_2 T_2)-d_2\beta e^{4y_1-4y_2}},
\end{align*}
where $K(t)=2d_1d_2y_2'(t)+\frac{(d_2^2-d_2)(y_2'(t))^2-c_1}{y_1'(t)}$. Note that by making $\epsilon>0$ smaller if necessary, $K$ is bounded on $[1-\epsilon,1)$ 
because $y_1'$ is continuous on $[1-\epsilon,1)$, and $\lim_{t\to 1}y_1'(t)=-\infty$. 

We rearrange:
\begin{equation}
\begin{aligned}
 -z_1'&=\frac{z_1}{(1-t)}+
 \frac{\left(d_1(d_1-1)\frac{z_1^2}{(1-t)}-Kz_1\right)\left(d_1-1-c_1T_1+\frac{z_1}{(1-t)}(\frac{d_1T_1'}{2}+\frac{d_2T_2'c_2e^{2y_1-2y_2}}{2c_1})+\beta\frac{d_2}{d_1}e^{4y_1-4y_2}\right)}{d_1(d_1-1-c_1T_1)+d_2 e^{2y_1-2y_2}(\alpha-c_2 T_2)-d_2\beta e^{4y_1-4y_2}}\\
 &=\frac{z_1}{1-t}\left(1+\frac{\left(d_1(d_1-1)z_1-K(1-t)\right)\left(d_1-1-c_1T_1+\frac{z_1}{(1-t)}(\frac{d_1T_1'}{2}+\frac{d_2T_2'c_2e^{2y_1-2y_2}}{2c_1})+\beta\frac{d_2}{d_1}e^{4y_1-4y_2}\right)}{d_1(d_1-1-c_1T_1)+d_2 e^{2y_1-2y_2}(\alpha-c_2 T_2)-d_2\beta e^{4y_1-4y_2}}\right)\\
 &=\frac{z_1}{1-t}\left(g_0(t)+g_1(t)z_1+g_2(t)z_1^2\right),
\end{aligned}
\end{equation}
where
\begin{equation}
\begin{aligned}
 g_0(t)&=1-\frac{K(1-t)(d_1-1-c_1T_1+\beta\frac{d_2}{d_1}e^{4y_1-4y_2})}{d_1(d_1-1-c_1T_1)+d_2 e^{2y_1-2y_2}(\alpha-c_2 T_2)-d_2\beta e^{4y_1-4y_2}},\\
 g_1(t)&=\frac{d_1(d_1-1)(d_1-1-c_1T_1+\beta\frac{d_2}{d_1}e^{4y_1-4y_2}))-K(\frac{d_1T_1'}{2}+\frac{d_2T_2'c_2e^{2y_1-2y_2}}{2c_1})}{d_1(d_1-1-c_1T_1)+d_2 e^{2y_1-2y_2}(\alpha-c_2 T_2)-d_2\beta e^{4y_1-4y_2}},\\
 g_2(t)&=\frac{d_1(d_1-1)(\frac{d_1T_1'}{2(1-t)}+\frac{d_2T_2'c_2e^{2y_1-2y_2}}{2c_1(1-t)})}{d_1(d_1-1-c_1T_1)+d_2 e^{2y_1-2y_2}(\alpha-c_2 T_2)-d_2\beta e^{4y_1-4y_2}}.\\
\end{aligned} 
\end{equation}

Using the smoothness conditions for $T$ along with Lemma \ref{0OSC}, we see that for each $i=0,1,2$, the function $g_i$ can be continuously extended all the way up to $1$, with $g_0(1) = 1$, $g_1(1)=d_1-1$, $g_2(1)=-d_1$. It is then convenient to write this as
\begin{align}\label{BUEFZ}
 z_1'=\frac{z_1}{(1-t)}\left((z_1-1)+h_0(t)+d_1z_1(z_1-1)+h_1(t)z_1+h_2(t)z_1^2\right),
\end{align}
where $h_i$ are continuous functions, with $h_i(1)=0$. From this we can see that $\lim_{t\to 1}z_1(t)=1$, for if not, there exists an $\epsilon' >0$ so that $z_1(t)$ achieves values greater than $1+\epsilon'$ (or lower than $1-\epsilon'$) for certain values of $t$
arbitrarily close to $1$. In the first case, \eqref{BUEFZ} implies that $\lim_{t\to T}z_1(t)=\infty$ for some $T<1$ because of the non-linear $z_1(z_1-1)$ term. 
In the second case, we see that there is a constant $C'''>0$ so that $z_1'(t)\le \frac{-C''' z_1(t)}{1-t}$, so we see that $\lim_{t\to 1}z_1(t)=0$, 
so eventually, $y_1'(t)\ge -\frac{1}{4(1-t)}$. 
By the first equation of \eqref{NEF2S'}, this in turn implies that eventually, $y_1''\ge 0$, in which case we see that $y_1'$ is bounded, which is a contradiction. Therefore, $\lim_{t\to 1}z_1(t)=1$, so by making $\epsilon>0$ smaller if necessary, we see that $y_1'(t)\le -\frac{0.75}{(1-t)}$, so $y_1(t) - y_1(1-\epsilon)\le 0.75\ln(1-t)-0.75\ln(1-\epsilon)$, showing that $\frac{e^{2y_1(t)}}{(1-t)^{1.5}}$ is bounded. 

The boundedness of $\frac{e^{2y_1(t)}}{(1-t)^{1.5}}$ implies that $\lim_{t\to 1}\frac{d_1}{2}e^{-2y_1}T_1'+\frac{d_2c_2T_2'}{2c_1}e^{-2y_2}=-\infty$, and 
we already know that $\lim_{t\to 1}e^{2y_1-4y_2}=0$, so the second
equation of \eqref{NEF2S'} implies that either $\lim_{t\to 1}y_2'(t)=0$, or eventually  $y_2''$ has a sign, so $\lim_{t\to 1}y_2'(t)$ exists and is a negative real number 
since $y_2'$ is bounded.
This implies that $K$ is continuous on $[1-\epsilon,1]$, so we can conclude, using the smoothness conditions for $T$, that the continuous functions $g_i$ (and hence $h_i$) are actually differentiable at $1$.

Therefore, \eqref{BUEFZ} can be written as 
\begin{align}\label{BUEFZ'}
 z_1'(t)=\frac{a(t)(z_1(t)-1)}{1-t}+b(t),
\end{align}
where $a$ and $b$ are both continuous on $[1-\epsilon,1]$, and $a(1)=d_1+1>0$. Proposition \ref{TCs} implies that $z_1\in C^1[1-\epsilon,1]$.
Therefore, we get 
$y_1'(t)=\frac{1}{t-1}+O(1)$ with $O(1)$ continuous on $[1-\epsilon,1]$, so $\frac{1}{f_1(t)^2}=e^{-2y_1(t)}=\frac{C(t)}{(1-t)^2}$, 
where $C(t)$ is a continuously differentiable positive function. This shows that $f_1'$ exists and is continuous on $[1-\epsilon,1]$, with $f_1'(1)<0$. 
Then \eqref{NEF2S'H2} implies that $h^2(t)$ is convergent as $t\to 1$. In fact, by using $y_1'=\frac{f_1'}{f_1}$ and
$e^{y_1}=f_1$ and the boundedness of $\frac{f_2'}{f_2}$, \eqref{NEF2S'H2} implies that $f_1'(1)=-h(1)$. 
\end{proof}

\begin{lemma}\label{1OSC'}
 Let $(h,y_1,y_2)$ be as in Lemma \ref{1OSC}. Then $f_2\in C^1[1-\epsilon,1]$, and $f_2'(1)=0$.
\end{lemma}
\begin{proof}
We now know that
$\lim_{t\to 1}e^{2y_1-4y_2}=0$, $h$ is a continuous positive function, and $e^{-2y_1}T_1'$ blows up like $\frac{1}{t-1}$. 
Therefore, the function $y_2$ must satisfy $\lim_{t\to 1}y_2'(t)=0$ 
because otherwise, $y_2'(t)$ stays away from $0$ for values of $t$ arbitrarily close to $1$ so second equation of \eqref{NEF2S'} implies that $y_2'$ is not bounded, which is a contradiction. Then $y_2\in C^1[1-\epsilon,1]$ and $y_2'(1)=0$, so we also have $f_2\in C^1[1-\epsilon,1]$ with $f_2'(1) = 0$.  
\end{proof}

\begin{lemma}\label{SOSC}
 Let  $(h,y_1,y_2)$ be as in Lemma \ref{1OSC'}. Then $f_1,f_2\in C^2[1-\epsilon,1]$, $h\in C^1[1-\epsilon,1]$, $h'(1) = 0$ and $f_1''(1)=0$.
\end{lemma}
\begin{proof}
By the proof of Lemma \ref{1OSC} we know that $f_1(t)=(1-t)K(t)$ where $K$ is a positive and continuously differentiable function, so in particular, 
$f_1'(t)=-K(t)+(1-t)K'(t)$. 
Next,
 \begin{align}
  \lim_{t\rightarrow 1} \frac{f_1'(t) - f_1'(1)}{t-1} = \lim_{t\rightarrow 1} \frac{K'(t)(1-t) - K(t) - (-K(1))}{t-1} = -2K'(1)
 \end{align}
 Therefore $f_1''(1)$ exists and $f_1''(1) = -2K'(1)$. The existence of $f_1''(1)$ implies, by Taylor's Theorem, that 
 $\frac{1}{f_1(t)}=\frac{1}{-(1-t)f_1'(1) + (1-t)^2g(t)}$ for some $g\in C^0[1-\epsilon,1]$. The fact that 
 $T_1$ satisfies smoothness conditions then implies that
 $\frac{T_1'(t)}{2f_1(t)}=\frac{1}{f_1'(1)}+(1-t)\tilde{g}(t)$, where $\tilde{g}\in C^0[1-\epsilon,1]$. 

Now from \eqref{NEF2S'H2} and the existence of $f_1''(1)$, we see that 
$\frac{h(t)^2}{(f_1'(t))^2}=a(t)$ is continuous and $a'(1)$ exists, so the second equation of \eqref{NEF2S} can be written as
\begin{align}
f_1f_1''&=(f_1')^2\left(1+\frac{h^2(d_1-1)}{(f_1')^2}-\frac{h^2d_1T_1'}{2f_1'f_1}\right)-f_1'h^2\frac{f_1d_2c_2T_2'}{2c_1f_2^2}-h^2c_1T_1 + h^2\beta\frac{d_2}{d_1}\frac{f_1^4}{f_2^4}\\
&=(f_1')^2\frac{a(t)}{f_1'(1)}\left(\frac{f_1'(1)}{a(t)}+(d_1-1)f_1'(1)-f_1'd_1\right)+ b(t)(1-t),
\end{align}
where $b\in C^0[1-\epsilon,1]$. Since $a'(1)$ exists, this equation has the required form to apply Proposition \ref{TCs}, so 
we find that $f_1'\in C^1[1-\epsilon,1]$. 
We can also use Proposition \ref{TCs} on the second equation of \eqref{NEF2S'} to show that $y_2'$, and therefore $f_2'$, are both in $ C^1[1-\epsilon,1]$. The expression for $h^2$ then shows that $h\in C^1[1-\epsilon,1]$, and 
$h(1)h'(1)=f_1'(1)f_1''(1)$, so $h'(1)=-f_1''(1)$. Then the second equation of \eqref{NEF2S} coupled with the fact that 
$h^2=(f_1')^2+k(t)(1-t)$ for some $k\in C^0[1-\epsilon,1]$ with $k(1)=0$ implies that $f_1''(1)=0$, so the same is true for $h'(1)$.
\end{proof}

\begin{proof}[Proof of Theorem \ref{LS:R}]

 We know that $f_i\in C^0[1-\epsilon,1]$, with $f_1(1)=0$ and $f_2(1)>0$ by Lemma \ref{0OSC}. 
 Then Lemmas \ref{1OSC} and \ref{1OSC'} imply that $f_i\in C^1[1-\epsilon,1]$ and $h\in C^0[1-\epsilon,1]$, with 
 $f_2'(1)=0$ and $f_1'(1)=-h(1)$. Finally, Lemma \ref{SOSC} implies that $f_i\in C^2[1-\epsilon,1]$ and $h\in C^1[1-\epsilon,1]$ with $
  f_1''(1)=0$ as required. 
 \end{proof}
\begin{proof}[Proof of Theorem \ref{RicciRegularity}, with \eqref{WeakRegularity} replaced by \eqref{StrongRegularity}]
By Proposition \ref{EOIE} and the discussion following it, we can assume that the equations \eqref{NEF2S'} and \eqref{NEF2S'H2} are satisfied. 
Since \eqref{StrongRegularity} holds, Theorem \ref{LS:R} implies that \eqref{SCST} holds. 
Then Theorem \ref{C1RG} implies that $\g$ is $C^1$ on all of $M$.

Now, the metric $\g$ is $C^\infty$ on the principal part of the manifold. Therefore, 
since $T>0$ is positive and smooth, Theorem \ref{DTG} in Appendix \ref{DKRE} (which is similar to Theorem 4.5 (a) of \cite{DeTurckKazdan}) 
implies that the Riemannian metric $\g$ is automatically smooth on the \textit{entirety} 
of $M$ since it is $C^1$ on all of $M$.
\end{proof}

\section{Proof of Theorem \ref{RicciExistence} in the case that $T_2$ is constant}\label{T2CC}
We now turn to the task of proving Theorem \ref{RicciExistence}. Since we assume that 
$[\mathfrak{n}_1,\mathfrak{n}_2]=0$, we know that $\beta=0$. We first prove Theorem \ref{RicciExistence} in the case $T_2>0$ is constant. We do this for a number of reasons. 
Firstly, in this case, we are actually able to show that the choice of $c_2$ is unique, hence motivating the need to allow such additional scaling in the first place. Secondly, 
the proof simplifies dramatically, but not to the extent that it bears no resemblance to the more general proof, so we can treat this step as a warm up. 
In fact, the proof in this case is essentially identical to the proof of Theorem 6.3 in \cite{Hamilton84}, so we are afforded the opportunity to review 
(and fill in some 
omitted computations) in Hamilton's work.

In the case that $T_2$ is constant, the equations of \eqref{PRC2S''} simplify to those studied by Hamilton in \cite{Hamilton84} because of the following proposition. 
\begin{prop}\label{MPF2}
Let $\beta = 0$. Suppose that $T_2>0$ is constant and $(h,f_1,f_2, c_1, c_2)$ is a solution of \eqref{PRC2S''} on $[\frac{1}{2},1]$ satisfying \eqref{SC} and \eqref{EC}, and $h, f_2>0$. Then the constant $c_2$ is determined as $c_2 = \frac{\alpha}{T_2}$, and $f_2>0$ is constant. 
\end{prop}
\begin{proof}
Observe that the third equation of \eqref{PRC2S''} can be written 
\begin{align*}
 f_2''(t)=a(t)f_2'(t)+b(t),
\end{align*}
where $a$ is continuous on $[\frac{1}{2},1)$, $\lim_{t\to 1}a(t) = +\infty$ and $b(t):=\frac{h^2(\alpha-c_2T_2)}{f_2}$ is continuous on $[\frac{1}{2},1]$. 
Since the function $b(t)$ does not change sign on $[\frac{1}{2},1]$ unless it is uniformly zero, the Neumann
 conditions $f_2'(\frac{1}{2})=f_2'(1)=0$ imply that $b(t)$ must be uniformly zero. Then clearly $f_2'(t)=0$, and $\alpha-c_2T_2=0$.  
\end{proof}
Thus if $T_2>0$ is constant, Proposition \ref{MPF2} implies that $f_2$ is constant, and we are left with the first 
two equations of \eqref{PRC2S''}:
\begin{align}\label{PRC2S1S}
\begin{split}
d_1\frac{f_1''}{f_1}-d_1\frac{h'f_1'}{hf_1}&=-c_1\\
 \frac{f_1''}{h^2f_1}+(d_1-1)\frac{(f_1')^2}{h^2f_1^2}-\frac{h'f_1'}{h^3 f_1}-\frac{1}{f_1^2}(d_1-1)&=-\frac{c_1T_1}{f_1^2},
 \end{split}
\end{align}
alongside the smoothness conditions
\begin{align}
 f_1(1)=0,&\qquad h(1)>0,\\
 f_1'(1)=-h(1)<0, &\qquad f_1''(1)=h'(1) = 0,
\end{align} 
as well as the evenness conditions
\begin{align}
 h'(\tfrac{1}{2})=0,\qquad
 f_1'(\tfrac{1}{2})&=0.
\end{align}
\subsection{The solution for large $c_1$}
Following Hamilton, we define the function $l(t)=\frac{-f_1'(t)}{h(t)}$. Then provided \eqref{PRC2S1S}, holds, we have
\begin{align}\label{EFL1S}
 l'(t)=\sqrt{\frac{c_1}{d_1}}\sqrt{c_1T_1(t)+(d_1-1)l^2(t)-(d_1-1)},
\end{align}
and the even and smoothness conditions become $l(\frac{1}{2})=0$, $l(1)=1$.

Now since $T_1>0$ on $[\frac{1}{2},1)$, if $c_1$ is large enough, then we can always solve \eqref{EFL1S} for $l$ with $l(\frac{1}{2})=0$. Indeed, 
the growth of $\sqrt{\frac{c_1}{d_1}}\sqrt{c_1T_1(t)+(d_1-1)l^2(t)-(d_1-1)}$ in $l$ does not exceed linear, and the only other obstruction to existence comes from the square root, but having 
$c_1$ large ensures that 
\begin{align}\label{PCSR}
c_1T_1(t)+(d_1-1)l^2(t)-(d_1-1)>0 \ \text{for} \ t\in [\tfrac{1}{2},1]. 
\end{align}
\subsection{Lowering $c_1$ until we achieve the correct terminal smoothness condition}
Let $\hat{c}_1$ be the infimum 
of all values of $c_1>0$ so that a solution of \eqref{EFL1S} with $l(\frac{1}{2})=0$ exists on $[\frac{1}{2},1]$, 
and \eqref{PCSR} holds. Clearly $\hat{c}_1>0$ because at $t=\frac{1}{2}$, $c_1T_1(t)+(d_1-1)l^2(t)-(d_1-1)=c_1T_1(\frac{1}{2})-(d_1-1)$ which becomes negative for small enough 
$c_1$.

Now by taking a sequence of $c_1$ convergent to $\hat{c}_1$ from above, we obtain a corresponding sequence of solutions $l(t)$, convergent to another solution
$\hat{l}(t)$, except now, there is a point $t^*$ with 
\begin{align*}
 \hat{c}_1T_1(t^*)+(d_1-1)\hat{l}^2(t^*)-(d_1-1)=0.
\end{align*}

If $t^*=1$, then we have $\hat{l}(1)=1$ as required. 
If $t^*\in (\frac{1}{2},1)$, then $\hat{c}_1T_1'(t^*)+2(d_1-1)\hat{l}'(t^*)\hat{l}(t^*)=0$, but since $\hat{c}_1T_1(t^*)+(d_1-1)\hat{l}^2(t^*)-(d_1-1)=0$, we have 
$\hat{l}'(t^*)=0$, so we obtain 
$T_1'(t^*)=0$, which is also a contradiction. 
Finally, if $t^*=\frac{1}{2}$ (this case was neglected by Hamilton), then $\hat{c}_1T_1(\frac{1}{2})-(d_1-1)=0$, so $\hat{c}_1 T_1-(d_1-1)\le 0$ on $[\frac{1}{2},1]$. 
Therefore, $\hat{l}$, which is a continuously differentiable and non-negative function, satisfies $\hat{l}'(t)\le \sqrt{\frac{\hat{c}_1(d_1-1)}{d_1}}\hat{l}(t)$, so $\hat{l}(t)=0$, which is clearly a contradiction. 

We therefore find a solution of $\hat{l}$ to \eqref{EFL1S} with the boundary conditions $\hat{l}(\frac{1}{2})=0$, $\hat{l}(1)=1$. 
We can then use the function $\hat{l}$ to find the functions $h$ and $f_1$ (and $f_2$) solving \eqref{PRC2S''}. Since $\hat{l}'(1) = 0$, 
this solution satisfies the hypotheses of Theorem \ref{RicciRegularity} (with \eqref{WeakRegularity} replaced by \eqref{StrongRegularity}), 
so we conclude that the smoothness conditions are satisfied.
\section{An Approximating Sequence of Solutions}\label{AASOS}
We now turn to the task of proving Theorem \ref{RicciExistence} in the case that $T_2$ is not constant. 
The proof consists of two main steps. The first step involves producing a sequence of solutions to the prescribed Ricci curvature equation that `almost' satisfies the hypotheses of Theorem \ref{RicciRegularity}. The second step involves taking a limit of this sequence, and showing that the limiting solution
does indeed satisfy 
the hypotheses of Theorem \ref{RicciRegularity} (with \eqref{StrongRegularity} instead of \eqref{WeakRegularity}). 
In this section, we prove Theorem \ref{TESS} below, which accomplishes the first step.
\begin{theorem}\label{TESS}
 Fix $a,S>0$. There exists a sequence of smooth functions $(h^{(k)}>0,y_1^{(k)},y_2^{(k)})$ on $[\frac{1}{2},1]$, and a corresponding sequence of bounded constants 
 $c_1^{(k)}>\frac{d_1-1}{T_1(\frac{1}{2})}$, 
 $c_2^{(k)}\in (\frac{\alpha}{\sup T_2},\frac{\alpha}{\inf T_2})$ solving \eqref{NEF2S'} (with $\beta = 0$) such that:
 \begin{itemize}
  \item $(y_2^{(k)})'(1)=0$ (smoothness condition for $f_2$),
  \item $(y_i^{(k)})'\le 0$,
  \item $(h^{(k)})'(\frac{1}{2})=(y_1^{(k)})'(\frac{1}{2})=(y_2^{(k)})'(\frac{1}{2})=0$ (evenness around $\frac{1}{2}$),
  \item $e^{-2y^{(k)}_1}(\frac{1}{2})=a$, $\sup_{t\in [\frac{1}{2},1]} e^{-2y^{(k)}_2(t)}+(y_2^{(k)})'(t)=S$,
  \item $(h^{(k)},y_1^{(k)},y_2^{(k)})$ is unbounded in $C^0[\frac{1}{2},1]\times C^1[\frac{1}{2},1]\times C^1[\frac{1}{2},1]$. 
 \end{itemize}
\end{theorem}
\subsection{Re-writing The Equations}\label{Rewriting}
For producing the required sequence of solutions, it will be useful to work with the functions $v_i=\ln(h)-y_i$. Note that if 
$(h,y_1,y_2)$ solves \eqref{NEF2S'} and is bounded in 
$C^0[\frac{1}{2},1]\times C^1[\frac{1}{2},1]\times C^1[\frac{1}{2},1]$, then $\ln(h)$ is also bounded in $C^1[\frac{1}{2},1]$, so $v_i$ is also 
bounded in $C^1[\frac{1}{2},1]$. So for our purposes, it suffices to produce a sequence of solutions to \eqref{NEF2S'} 
so that the corresponding $v_i$ functions are unbounded in 
$C^1[\frac{1}{2},1]$.

Now observe that $he^{-y_i}=e^{v_i}$, so we can differentiate and use \eqref{NEF2S'}:
\begin{equation}
\begin{aligned}
 v_i'&=\frac{h'}{h}-y_i',\\
 &=(d_1y_1'+d_2y_2'-y_i')-\tfrac{d_1}{2}T_1'e^{2v_1}-\tfrac{d_2c_2}{2c_1}T_2'e^{2v_2}.
\end{aligned} 
\end{equation}
Differentiating again we can write second order differential equations for $v_1$ and $v_2$:
\begin{equation}\label{TMEFV}
\begin{aligned}
 v_i''&=(d_1y_1''+d_2y_2''-y_i'')-\tfrac{d_1}{2}T_1''e^{2v_1}-d_1T_1'e^{2v_1}v_1'-\tfrac{d_2c_2}{2c_1}T_2''e^{2v_2}-\tfrac{d_2c_2}{c_1}T_2'e^{2v_2}v_2'\\
 &=(d_1-\delta_1^i)(d_1-1-c_1T_1)e^{2v_1}+(d_2-\delta_2^i)(\alpha-c_2T_2)e^{2v_2}-\tfrac{d_1}{2}T_1''e^{2v_1}-\tfrac{d_2c_2}{2c_1}T_2''e^{2v_2}\\
 &-d_1T_1'e^{2v_1}v_1'-\tfrac{d_2c_2}{c_1}T_2'e^{2v_2}v_2'\\
 &=(d_1-\delta_1^i)(d_1-1-c_1T_1)e^{2v_1} +(d_2-\delta_2^i)(\alpha-c_2T_2)e^{2v_2}-\tfrac{d_1}{2}T_1''e^{2v_1}-\tfrac{d_2c_2}{2c_1}T_2''e^{2v_2}\\
 &-(v_i'+\tfrac{d_1}{2}T_1'e^{2v_1}+\tfrac{d_2c_2}{2c_1}T_2'e^{2v_2})\left(\tfrac{d_1}{2}T_1'e^{2v_1}+
 \tfrac{d_2c_2 }{2c_1}T_2'e^{2v_2}\right)-d_1T_1'e^{2v_1}v_1'-\tfrac{d_2c_2}{c_1}T_2'e^{2v_2}v_2'.
\end{aligned} 
\end{equation}
We write these equations as
\begin{equation}\label{SOEFV}
 v_i''(t)=F_i(t,c_1,c_2,e^{2v_1}(t),e^{2v_2}(t),v_1'(t),v_2'(t)).
\end{equation}
Note that the first and second equations of \eqref{NEF2S'} can be written as
\begin{align}
\begin{split}
\label{NEF2S'V}
y_1''&=\left((d_1-1-c_1T_1)e^{2v_1}-y_1'\left(\tfrac{d_1}{2}T_1'e^{2v_1}+\tfrac{d_2c_2}{2c_1}T_2'e^{2v_2}\right)\right),\\
y_2''&=\left((\alpha-c_2T_2)e^{2v_2}-y_2'\left(\tfrac{d_1}{2}T_1'e^{2v_1}+\tfrac{d_2c_2}{2c_1}T_2'e^{2v_2}\right)\right).\\
\end{split}
\end{align}
If we happen to have a solution $(v_1,v_2)$ of \eqref{SOEFV}, then \eqref{NEF2S'V} can be solved for $y_i'$ in terms of $v_i$ since we know 
$y_i'(\frac{1}{2})=0$. Similarly, integrating one more time yields $y_i$ if we know $y_i(\frac{1}{2})$. Then $h$ can be found with $h=e^{v_i+y_i}$. Thus, the original functions $y_i$, $h$ can be recovered if we know $v_i$.

We will now express the problem of finding solutions of \eqref{SOEFV} (with appropriate boundary conditions) in terms of finding zeroes of a suitably defined function $F$. First, we introduce the notation needed for this.
Define $e^{-2y_2(\frac{1}{2})}=\gamma$. We will require our solutions to satisfy $e^{-2y_1}(\frac{1}{2})=a$ and $\sup_{t\in [\frac{1}{2},1]} e^{-2y_2(t)}+(y_2'(t))^2=S$. Then using \eqref{NEF2S'}, we see that
\begin{equation}\label{ICH}
\displaystyle
h\left(\tfrac{1}{2}\right)=\sqrt{\frac{-c_1}{d_1a(d_1-1-c_1T_1(\frac{1}{2}))+d_2\gamma(\alpha-c_2T_2(\frac{1}{2}))}}.
\end{equation}
We write this as $h(\frac{1}{2})=I(a,\gamma,c_1,c_2)$. 
Since $T_i'\le 0$, we see that $T_2(\frac{1}{2})=\sup T_2$, so for fixed $a$, the function $I$ is continuous in $\gamma\ge 0$, $c_2\in [\frac{\alpha}{\sup T_2},\frac{\alpha}{\inf T_2}]$ and $c_1>\frac{d_1-1}{T_1(\frac{1}{2})}$. Moreover, there exists a number $C>0$ depending on $S,a,T_1$ and $T_2$ so that $\frac{1}{2}\sqrt{\frac{1}{d_1 aT_1(\frac{1}{2})}}\le I(a,\gamma,c_1,c_2)\le 2\sqrt{\frac{1}{d_1aT_1(\frac{1}{2})}}$ for all $c_1\ge C$, $\gamma\le S$ and $c_2\in  [\frac{\alpha}{\sup T_2},\frac{\alpha}{\inf T_2}]$. 
We then obtain the following initial conditions for $v_1$ and $v_2$:
\begin{align}\label{ICFV}
\begin{split}
 \frac{e^{v_1(\frac{1}{2})}}{\sqrt{a}}=\frac{e^{v_2(\frac{1}{2})}}{\sqrt{\gamma}}&=I(a,\gamma,c_1,c_2),\\
 v_1'(\tfrac{1}{2})=v_2'(\tfrac{1}{2})&=0.
 \end{split}
\end{align}
Now, we define a function
$F:C^1([\tfrac{1}{2},1]:\mathbb{R}^2)\times [\tfrac{\alpha}{\sup T_2},\tfrac{\alpha}{\inf T_2}]\times  
(0,S+1]\to C^1([\tfrac{1}{2},1]:\mathbb{R}^2)\times \mathbb{R}^2$ as follows:
\begin{align}
F(v_1,v_2,c_2,\gamma)=(v_1-w_1,v_2-w_2,y_2'(1),(1-p_2)\gamma+p_2\sup_{[\frac{1}{2},1]}(e^{-2y_2}+(y_2')^2)-S),
\end{align}
where $w_1,w_2$ and $y_2$ are chosen to satisfy the following:
\begin{equation}
\begin{aligned}\label{eqn:defnF}
 w_i''(t)&=p_1F_i(t,c_1,c_2,e^{2v_1}(t),e^{2v_2}(t),v_1'(t),v_2'(t)),\\
 w_1(\tfrac{1}{2})&=p_3\ln(\sqrt{a}I(a,\gamma,c_1,c_2)),\quad w_2(\tfrac{1}{2})=p_3\ln(\sqrt{\gamma}I(a,\gamma,c_1,c_2)),\quad w_i'(\tfrac{1}{2})=0,\\
 y_2''&=e^{2p_3v_2}\left((\alpha-c_2T_2)-p_3y'_2\left(\tfrac{d_1}{2}T_1'e^{2v_1-2v_2}+\tfrac{d_2c_2 }{2c_1}T_2'\right)\right),\\
 y_2(\tfrac{1}{2})&=-\frac{\ln(\gamma)}{2}, \qquad y_2'(\tfrac{1}{2})=0.
\end{aligned}
\end{equation}
Then, solving \eqref{SOEFV} with the initial conditions \eqref{ICFV}, $y_2'(1)=0$ and $\sup_{[\frac{1}{2},1]}(e^{-2y_2}+(y_2')^2)=S$ 
is equivalent to finding the zeroes of the function $F$ when the parameters $p_1$, $p_2$, and $p_3$ are all set equal to $1$. Note that if $V$ is an open and bounded ball in 
$C^1([\frac{1}{2},1]:\mathbb{R}^2)$ and $\overline {\gamma} \in (0,S+1)$, then $F$ is a compact perturbation of the 
identity on $\overline{\Omega}$, where \begin{equation}\label{defOmega}
 \Omega=V\times (\tfrac{\alpha}{\sup T_2},\tfrac{\alpha}{\inf T_2})\times  
(\overline{\gamma},S+1).
\end{equation} We know from maximum principle arguments similar to those in Proposition \ref{MPF2} 
that if $(v_1,v_2,c_2,\gamma)\in C^1([\frac{1}{2},1]:\mathbb{R}^2)\times [\frac{\alpha}{\sup T_2},\frac{\alpha}{\inf T_2}]\times  
(0,S+1]$ is a zero of $F$, then $c_2\in(\frac{\alpha}{\sup T_2},\frac{\alpha}{\inf T_2})$. It is also clear that $\gamma\in(0,S]$. 
Note that to prove Theorem \ref{TESS}, it suffices to construct an unbounded sequence of zeroes of $F$ so that $y_1'\le 0$ and $y_2'\le 0$. 
\subsection{The solution for large $c_1$}
Having re-written the equations of \eqref{NEF2S'}, we are now
in a position to construct the sequence of solutions satisfying the conclusion of Theorem \ref{TESS}. 
The main idea is essentially what we already saw in Section \ref{T2CC}: find a $c_1$ large enough so that we actually have a solution, and then lower $c_1$ until we encounter an obstruction to existence. In this subsection, we state Theorem \ref{PDT} which deals with existence of solutions for $c_1$ large. We also state and prove Propositions \ref{BasicEstimates} and \ref{CSUBS}, which establish existence of the required sequence of solutions. The proof of Theorem \ref{PDT} makes use of several technical lemmas, which are proved in the next subsection.
\begin{theorem}\label{PDT}
 Set $p_1=p_2=p_3=1$. There exists a large $(c_1)^*$ so that when $c_1=(c_1)^*$, all of the zeroes of $F$ lie in a bounded set $\Omega$ of the form
\eqref{defOmega}. 
Furthermore, $\deg(F,\Omega,0)\neq 0$. 
\end{theorem}
\begin{proof}
The main idea is to use topological degree theory and continuously deform the problem with all $p_i$'s equal to $1$, to the problem with all $p_i$'s equal to $0$. If we can show that solutions remain bounded under this deformation, then we can use information about the latter problem to deduce information about the 
former problem.

First, we select the value $(c_1)^*$ given by Lemma \ref{SABLC}, and work with this value for the rest of the proof. 
If $p_1 = p_2 = p_3 = 0$, then we see that $F(v_1, v_2, c_2, \gamma) = (v_1, v_2,\int_{\frac{1}{2}}^{1}(\alpha-c_2T_2), \gamma - S)$.
The unique zero of $F$ is $(v_1,v_2,c_2,\gamma)=(0,0,\frac{\alpha}{2\int_{\frac{1}{2}}^{1}T_2},S)$, so it is clear that there exists an $\Omega$ having the 
form \eqref{defOmega} so that all of the zeroes of $F$ with $p_1=p_2=p_3=0$ lie in $\Omega$. 
Now we have $(I-F)(v_1,v_2,c_2,\gamma)=(0,0,c_2-\int_{\frac{1}{2}}^{1}(\alpha-c_2T_2),S)$; the only non-zero eigenvalue of the linearisation of this function is $1+\int_{\frac{1}{2}}^1 T_2>1$, so point (iii) of Theorem \ref{PLSD} implies that $\deg(F, \Omega, 0) \neq 0$ when $p_1 = p_2 = p_3 = 0$. 
Then, we use Lemmas \ref{BVP3'}, \ref{BVP2} and \ref{SABLC} to continuously deform $F$ while possibly enlarging $\Omega$ if required to contain all of the zeroes 
 (while still retaining the form \eqref{defOmega}). The homotopy invariance of the Leray-Schauder degree (point (ii) of Theorem \ref{PLSD}) 
 implies that $\deg(F,\Omega,0)$ for $p_1 = p_2 = p_3 = 1$ is the same as for $p_1 = p_2 = p_3 = 0$. In particular, $\deg(F,\Omega,0) \neq 0$ when $p_1 = p_2 = p_3 = 1$. This completes the proof.
\end{proof}

Since $\deg(F,\Omega,0) \neq 0$, point (i) of Theorem \ref{PLSD} gives us existence of zeroes of 
$F$ when $p_1 = p_2 = p_3 = 1$. In fact, in this case we obtain an entire family of zeroes, which are parametrized by $c_1$. We make the following observations about all of these solutions.
\begin{prop}\label{BasicEstimates}
 Suppose $p_1=p_2=p_3=1$, and $(v_1,v_2,c_2,\gamma)$ is a zero of $F$ with $c_1>\frac{d_1-1}{T_1(\frac{1}{2})}$. Then for the associated functions $y_1$ and $y_2$, we have $y_1'\le 0$, $y_2'\le 0$, and $e^{-2y_2}\le \gamma e^{\sqrt{S}}$.
\end{prop}
\begin{proof}
Since $T_2'\le 0$, the quantity $\alpha-c_2T_2$ is increasing on $[\frac{1}{2},1]$ and initially negative since $c_2\in (\frac{\alpha}{\sup T_2},\frac{\alpha}{\inf T_2})$. Therefore, the second equation of \eqref{NEF2S'} implies that $y_2'(t)$ becomes negative for some times $t>\frac{1}{2}$. Therefore, if $y_2'$ ever became positive afterwards, it would have to occur when $\alpha-c_2T_2$ is positive. But since $\alpha-c_2T_2$ then stays positive, we find that $y_2'$ stays positive, which contradicts $y_2'(1)=0$. Therefore $y_2'\le 0$. 

Now the first equation of
\eqref{PRC2S''} becomes 
\begin{align}\label{prctry}
 tr(y'')=-tr((y')^2)+\frac{h'}{h}tr(y')-c_1,
\end{align}
which implies that $d_1y_1'+d_2y_2'$ is negative on $(\frac{1}{2},1]$. 
Therefore, $y_1'(1)<0$ since $y_2'(1)=0$. However, $T_1'\le 0$ and the quantity $d_1-1-c_1T_1$ is also increasing and initially negative, so the same argument as for 
$y_2'$ can be used to show that if $y_1'$ were ever positive, it would have to stay that way up until $t=1$, which contradicts $y_1'(1)<0$. 

Finally, since $(y_2')^2 \le S$ and $e^{-2y_2 (\frac{1}{2})}=\gamma$, we have 
\begin{align*}
 e^{-2y_2(t)}=e^{2y_2(\frac{1}{2})-2y_2(t)-2y_2(\frac{1}{2})} =\gamma e^{-2\int_{\frac{1}{2}}^t y_2'(s)ds} \le \gamma e^{\sqrt{S}}.
\end{align*}
\end{proof}
We now construct our unbounded sequence of solutions. 
\begin{prop}\label{CSUBS}
 There exists a sequence of $c_1\in (\frac{d_1-1}{T_1(\frac{1}{2})}, c_1^*]$, which we denote $c_1^{(k)}$, convergent to $\hat{c}_1$, and a corresponding sequence of zeroes $(v_1^{(k)}, v_2^{(k)}, c_2^{(k)}, \gamma^{(k)})$ of $F$ (with $p_1=p_2=p_3=1$) so that $(v_1^{(k)}, v_2^{(k)})$ is unbounded in $C^1([\frac{1}{2},1]:\mathbb{R}^2)$.
\end{prop}
\begin{proof}
If this were not the case, then all zeroes of $F$ have $(v_1,v_2)$ bounded in $C^1([\frac{1}{2},1]:\mathbb{R}^2)$, 
independently of $c_1\in (\frac{d_1-1}{T_1(\frac{1}{2})}, c_1^*]$. 
These bounds imply that $\gamma$ is bounded away from $0$ uniformly by \eqref{ICFV}. We already know that $\gamma\le S$ and $c_2\in (\frac{\alpha}{\sup T_2},\frac{\alpha}{\inf T_2})$, so these bounds imply that 
there is an $\Omega$ having the form of \eqref{defOmega} so that all of the zeroes of $F$ (if any) lie in $\Omega$ 
for all $c_1\in (\frac{d_1-1}{T_1(\frac{1}{2})},c_1^*]$. 
Then the homotopy invariance of the Leray-Schauder degree implies that $\deg(F,\Omega,0)\neq 0$ for all $c_1\in (\frac{d_1-1}{T_1(\frac{1}{2})},c_1^*]$ since we know the degree is non-zero for $c_1=c_1^*$ by Theorem \ref{PDT}. 
In particular, there exists a solution lying in $\Omega$ for each such $c_1$.

Consider a sequence of $c_1^{(k)}$, convergent to $\frac{d_1-1}{T_1(\frac{1}{2})}$, and the corresponding zeroes of $F$. 
Since $a$ is fixed and $\liminf \gamma^{(k)}>0$, \eqref{NEF2S'V} implies that $(y_1^{(k)})$ and $(y_2^{(k)})$ are both uniformly bounded in $C^2[\frac{1}{2},1]$. 
Then $h^{(k)}=e^{y_1^{(k)}+v_1^{(k)}}$ is also bounded uniformly in $C^1[\frac{1}{2},1]$. 
Then by taking the limit as $c_1^{(k)}$ approaches $\frac{d_1-1}{T_1(\frac{1}{2})}$, we would obtain uniform convergence to a solution $(h,y_1,y_2)$ of \eqref{NEF2S'} 
with $c_1T_1\le d_1-1$ on the entirety of $[\frac{1}{2},1]$ since $T_1'\le 0$. Then the first equation of \eqref{NEF2S'} implies $y_1'(1)>0$, a contradiction with the fact that $(y_1^{(k)})'(1) \le 0$ for each $k$ by Proposition \ref{BasicEstimates}. We therefore find our sequence of solutions 
$(v_1^{(k)}, v_2^{(k)}, c_2^{(k)}, \gamma^{(k)})$ with $(v_1^{(k)}, v_2^{(k)})$ unbounded.
\end{proof}

\begin{proof}[Proof of Theorem \ref{TESS}]
 By Section \ref{Rewriting}, we can equivalently find a sequence of zeroes of $F$ when $p_1 = p_2 = p_3 = 1$. This is achieved by Theorem \ref{PDT}, since $\deg(F,\Omega,0)\neq 0$. Proposition \ref{CSUBS} yields the unbounded sequence of solutions (by the discussion at the start of Section \ref{Rewriting}, unboundedness of $(v_1^{(k)}, v_2^{(k)})$ implies unboundedness for the corresponding sequence $(h^{(k)}, y_1^{(k)}, y_2^{(k)})$). That the required properties are satisfied by the sequence of functions, is established by Proposition \ref{BasicEstimates}.
\end{proof}

\subsection{Boundedness of solutions under deformation}
Now we list the technical results needed in the proof of Theorem \ref{PDT}. 
These results allow us to carry out the various steps in the deformation from the case where 
$p_1 = p_2 = p_3 = 0$, to the case we are interested in, that is, $p_1 = p_2 = p_3 = 1$.
\begin{lemma}\label{BVP3'}
 Set $p_1=p_2=0$. For each $c_1\ge C$, there exists a bounded set $\Omega$ of the form \eqref{defOmega} such that zeroes of $F$ (if any) lie in $\Omega$ for each $p_3\in [0,1]$. 
\end{lemma}
\begin{proof}
 Since $p_1=p_2=0$, zeroes of $F$ satisfy 
 \begin{align*}
  v_1(t)=p_3\ln(\sqrt{a}I(a,\gamma,c_1,c_2)), \quad v_2(t)=p_3\ln(\sqrt{\gamma}I(a,\gamma,c_1,c_2))\quad \mbox{ and } \gamma=S,
 \end{align*}
   so are bounded independently of $p_3\in[0,1]$. As before, it is still clear that $c_2\in  (\frac{\alpha}{\sup T_2},\frac{\alpha}{\inf T_2})$. 
\end{proof}
\begin{lemma}\label{BVP2}
 Set $p_1=0$, $p_3=1$. For each $c_1\ge C$, there exists a bounded set $\Omega$ of the form \eqref{defOmega} such that the zeroes of 
 $F$ (if any) lie in $\Omega$ for any $p_2\in [0,1]$. 
\end{lemma}
\begin{proof}
 Since $p_1=0$, the zeroes $(v_1,v_2,c_2,\gamma)$ of $F$ satisfy $v_1(t) = \ln(\sqrt{a}I(a, \gamma, c_1, c_2))$ and $v_2(t)=\ln(\sqrt{\gamma}I(a, \gamma, c_1, c_2))$ and we also note that $\gamma\le S$ (irrespective of $p_2\in [0,1]$). It follows that $e^{2v_1}$ and $e^{2v_2}$ are bounded independently of $p_2$.
 
 Furthermore, if there was a sequence of zeroes with 
 $\gamma$ converging to $0$, then $e^{2v_2} = \gamma I(a,\gamma,c_1,c_2)^2$ goes to $0$ uniformly, so by integrating the third equation of \eqref{eqn:defnF}, we see that $y_2'$ converges to $0$ uniformly on $[\frac{1}{2},1]$. Then $y_2(\frac{1}{2}) = -\frac{\ln (\gamma)}{2}\to \infty$ implies that $e^{-2y_2}$ converges 
 to $0$ uniformly.
 
 The fact that $\gamma$, $e^{-2y_2}$ and $y_2'$ all converge to $0$ (uniformly for $e^{-2y_2}$ and $y_2'$)
 contradicts $p_2\gamma+(1-p_2)\sup_{[\frac{1}{2},1]}(e^{-2y_2}+(y_2')^2)=S$. Therefore, $\gamma$ stays away from $0$, and $v_1,v_2$ are bounded. 
\end{proof}

Note that the above 
two results are valid for all values of $c_1\ge C$, and the bounded set $\Omega$ depends on the value of $c_1$. 
The last step in the deformation ($p_2=p_3 = 1$, while $p_1$ goes from $0$ to $1$) 
requires selecting a large enough value of $c_1$, and the arguments are more involved.
\begin{prop}\label{DILE}
There exists a $D>0$ independent of $c_1\ge C$, $\gamma\le S$, $c_2\in  [\frac{\alpha}{\sup T_2},\frac{\alpha}{\inf T_2}]$ and $p_1\in [0,1]$ so that the following holds:
if $F(v_1,v_2,c_2,\gamma)=0$ ($p_2=p_3=1$) and $v_i'\le 0$, $i=1,2$ on $[\frac{1}{2},t^*]$ for some $t^*\le 1$, then 
$\left|v_i'\right|=-v_i'< Dc_1$ on $[\frac{1}{2},t^*]$, $i=1,2$. 
\end{prop}
\begin{proof}
With the assumption that $v_i'\le 0$, we find that $e^{2v_i}\le e^{2v_i(\frac{1}{2})}\le \frac{4(a+S)}{d_1 aT_1(\frac{1}{2})}$, 
for all $c_1\ge C$, $\gamma\le S$ and $c_2\in  [\frac{\alpha}{\sup T_2},\frac{\alpha}{\inf T_2}]$. Then considering the vector $v=(v_1,v_2)$, the first two equations of $F(v_1,v_2,c_2,\gamma)=0$ can be written as
\begin{align}
v''=p_1Av'+p_1c_1 B,
\end{align}
where $A(t)$ is a continuous $2\times 2$ matrix and $B(t)$ is a continuous vector, both bounded independently of $(c_1,c_2,\gamma)$. The result then follows from Gronwall's inequality. 
\end{proof}
\begin{prop}\label{EFVI}
 For each $\epsilon >0$, there is a $c_1\geq C$ large enough so that on $[\frac{1}{2}, 1-\epsilon]$, $v_1'-v_2' \geq 0$ and $0 \leq -v_i' \leq Dc_1$.
\end{prop}
\begin{proof}
We define a quantity $E>0$, independent of $c_1$, $c_2$ and $\gamma$ with:
\begin{align}
 E=\sup_{\substack{\gamma\le S,\, c_1\ge C,\\c_2\in [\frac{\alpha}{\sup T_2},\frac{\alpha}{\inf T_2}]}} \left(d_1(d_1-1)+\frac{\gamma}{a}(d_2\sup \left|\alpha-c_2T_2\right| + \frac{d_2c_2}{2c_1}\sup \left|T_2''\right|)
 +\frac{d_1}{2}\sup \left|T_1''\right|+1\right).
\end{align}
For any $\epsilon>0$, choose $c_1$ large enough so that $\min _{[\frac{1}{2},1-\epsilon]}c_1T_1=c_1T_1(1-\epsilon)\ge E$ 
(note that here we use that $T_1>0$, $T_1'<0$ on $(\frac{1}{2}, 1)$). Suppose we have $F(v_1,v_2,c_2,\gamma)=0$ 
for some $p_1\in [0,1]$. Then we claim that on $[\frac{1}{2},1-\epsilon]$, we have $v_1'-v_2'\ge 0$ and $0\le -v_i'\le D c_1$.

If $p_1=0$, the claim obviously holds, so  we assume that $0<p_1\le 1$. 
Now since $v_1'=v_2'=0$ at $t=\frac{1}{2}$, there is a maximal interval $[\frac{1}{2},t^*]\subseteq [\frac{1}{2},1-\epsilon]$ so that 
$v_1'-v_2' \geq 0$ and $0 \leq -v_i' \leq Dc_1$ holds on 
$[\frac{1}{2},t^*]$. We aim to show that $t^* = 1-\epsilon$. 
To see this, we can use the two characterizations of the quantity $E$ 
(as a supremum over $\gamma$, $c_1$, $c_2$ and as a minimum over $t\in [\frac{1}{2}, 1-\epsilon]$), and $t^*<1-\epsilon$, to find that on $[\frac{1}{2},t^*]$,
\begin{equation}\label{DVI'I}
\begin{aligned}
 v_1''-v_2''&\ge -p_1(v_1'-v_2')\left(\tfrac{d_1}{2}T_1'e^{2v_1}+\tfrac{d_2c_2 }{2c_1}T_2'e^{2v_2}\right)+p_1e^{2v_1}\left((c_1T_1-d_1+1)-e^{2v_2(\frac{1}{2})-2v_1(\frac{1}{2})}\sup \left|\alpha-c_2T_2\right|\right)\\
 &\ge -p_1(v_1'-v_2')\left(\tfrac{d_1}{2}T_1'e^{2v_1}+\tfrac{d_2c_2 }{2c_1}T_2'e^{2v_2}\right)+p_1e^{2v_1}\left((E-d_1+1)-\frac{\gamma}{a}\sup \left|\alpha-c_2T_2\right|\right)\\
 &\ge -p_1(v_1'-v_2')\left(\tfrac{d_1}{2}T_1'e^{2v_1}+\tfrac{d_2c_2 }{2c_1}T_2'e^{2v_2}\right)+p_1e^{2v_1}.
\end{aligned}
\end{equation}
Using in addition $0\le -v_i'\le Dc_1$ on $[\frac{1}{2},t^*]$, we also have
\begin{equation}\label{VI'I0}
 \begin{aligned}
 \frac{v_i''}{p_1}&\le (d_1-\delta_1^i)e^{2v_1}(d_1-1-c_1T_1)+(d_2-\delta_2^i)e^{2v_1+2v_2(\frac{1}{2})-2v_1(\frac{1}{2})}\sup \left|\alpha-c_2T_2\right|\\
 & \hspace{7cm}+\tfrac{d_1}{2}\left|T_1''\right|e^{2v_1}+\tfrac{d_2c_2}{2c_1}\left|T_2''\right|e^{2v_1+2v_2(\frac{1}{2})-2v_1(\frac{1}{2})}\\
 &\le e^{2v_1}\left(d_1(d_1-1)-d_1E+e^{2v_2(\frac{1}{2})-2v_1(\frac{1}{2})}(d_2\sup \left|\alpha-c_2T_2\right|+\tfrac{d_2c_2}{2c_1}\left|T_2''\right|)
 +\tfrac{d_1}{2}\left|T_1''\right|\right)<0.
\end{aligned}
\end{equation}
The inequalities \eqref{DVI'I} and \eqref{VI'I0} imply that $t^*>\frac{1}{2}$. 
In fact, these same inequalities show us that $v_i'<0$ and $v_1'-v_2'>0$ at $t^*$. Finally, since $v_i'\le 0$ on $[\frac{1}{2},t^*]$, we can use 
Proposition \ref{DILE} to conclude that $-v_i'(t^*)<Dc_1$. Therefore, we have a contradiction with the definition of $t^*$, unless $t^*=1-\epsilon$ as desired.
\end{proof}
\begin{lemma}\label{SABLC}
 Set $p_2=p_3=1$. There exists a large $(c_1)^*\ge C$ and an $\Omega$ having the form of \eqref{defOmega} so that when $c_1=(c_1)^*$, 
 the zeroes of $F$ (if any) lie in $\Omega$ for any $p_1\in [0,1]$. 
\end{lemma}
\begin{proof}
We need to find a $c_1$ and an $\Omega$ having the form of \eqref{defOmega} that contains all of the zeroes of $F$ for this choice of $c_1$ and any $p_1\in [0,1]$. Assume to the contrary that there is no such $\Omega$, 
so we have a sequence of zeroes of $F$ that cannot be appropriately contained. First note that since we have zeroes, we know that $\gamma\le S<S+1$, and it is also clear that $c_2\in  (\frac{\alpha}{\sup T_2},\frac{\alpha}{\inf T_2})$. Therefore, $\gamma\to 0$, or one of $v_1$ and $v_2$ become unbounded in $C^1[\frac{1}{2},1]$. 
In fact, if $\gamma\to 0$, then by construction, $v_2(\frac{1}{2})\to -\infty$, so we get unboundedness of $(v_1, v_2)\in C^1([\frac{1}{2},1]:\mathbb{R}^2)$ in either case. We will find a large $c_1>0$ so that this does not occur.

First, by Proposition \ref{DILE} and Proposition \ref{EFVI}, 
for $\epsilon$ small enough and $c_1$ large enough, we have $0\le -v_i'\le Dc_1$ on $[\frac{1}{2},1-\epsilon]$. We also claim that 
\begin{equation}\label{EFV'L}
-2Dc_1\le v_i'(t)\le 1
\end{equation}
for $t\in[1-\epsilon,1]$, provided $c_1$ is made larger, if necessary.
Certainly \eqref{EFV'L} holds for some interval $[1-\epsilon,t^*]$ for some $t^*\in (1-\epsilon,1]$. On such an interval, the estimates for 
$v_i'$ on $[\frac{1}{2},1-\epsilon]$ and $[1-\epsilon,t^*]$ imply that
\begin{align}\label{EFE2VL}
e^{2v_i}\le I(a,\gamma,c_1,c_2)^2 \max\{a,\gamma\}e^{2(t^*-(1-\epsilon))}\le \frac{4(a+S)}{d_1aT_1(\frac{1}{2})}e^{2(t^*-(1-\epsilon))}
\end{align}
on $[1-\epsilon,t^*]$. 
This estimate \eqref{EFE2VL}, alongside the fact that \eqref{EFV'L} holds for $t\in [1-\epsilon,t^*]$ implies that
\begin{align}
 v_i''(t)\le G_1, \ t\in [1-\epsilon,t^*],
\end{align}
for some $G_1$ independent of $c_1,c_2$ and $\gamma\le S$ since $T_i'\le 0$.
On the other hand, \eqref{EFE2VL} and \eqref{EFV'L} also imply 
\begin{align}
 -v_i''(t)\le G_2(1+c_1),\ t\in [1-\epsilon,t^*],
\end{align}
for some $G_2$ independent of $c_1,c_2$ and $\gamma$. Since $-Dc_1\le v_i'\le 0$ at $t=1-\epsilon$, the above estimates on $v_i''$ imply
\begin{align}
-Dc_1-G_2(1+c_1)\epsilon \le v_i'(t^*)\le G_1 \epsilon.
\end{align}

Since $G_1$, $G_2$ are independent of $c_1$ and $\gamma$, we can choose $\epsilon>0$ small enough (along with the corresponding value of $c_1$ provided 
by Proposition \ref{EFVI}) so that
\begin{align}
-2Dc_1< v_i'(t^*)<1,
\end{align}
which is a contradiction with the definition of $t^*$ being the last time \eqref{EFV'L} holds, unless $t^*=1$. Thus we conclude that $t^*=1$. 

We therefore find that the bounds \eqref{EFV'L} on $v_i'$ hold on $[\frac{1}{2}, 1]$, for some large enough $c_1$. We now fix such a value of $c_1$, 
so now, $v_i'$ is bounded independently of $p_1,c_2$ and $\gamma$. Suppose $\gamma\to 0$. Then the bounds on $v_2'$ and the initial conditions 
for $v_2$ imply that $e^{2v_2}$ converges to $0$ uniformly on $[\frac{1}{2},1]$, and $e^{2v_1}$ remains bounded. The definition of $y_2$ in \eqref{eqn:defnF} then implies that $e^{-2y_2}+(y_2')^2$ converges to $0$ uniformly, which is a contradiction.

Therefore $\gamma$ is bounded away from $0$. Then, since $a$ is fixed, $\gamma\le S$ and $v_i'$ is bounded, we 
can recover $C^1$ bounds on $v_i$, independently of $p_1$, concluding the proof.
\end{proof}
\section{Convergence}
With Theorem \ref{TESS} in hand, the proof of Theorem \ref{RicciExistence} will be complete if we can take a limit of the sequence of solutions, to arrive at 
a metric which satisfies the hypothesis of Theorem \ref{RicciRegularity} (with \eqref{StrongRegularity} instead of \eqref{WeakRegularity}). This is accomplished by Theorem \ref{LOKS} below. Although the proof is technical, the main idea
is to show that the solutions must remain bounded away from $t=1$, so we get smooth subconvergence away from the singular orbit. 
\begin{theorem}\label{LOKS}
 Suppose $(h^{(k)},y_1^{(k)},y_2^{(k)})$ is a sequence of functions satisfying \eqref{NEF2S'} (with $\beta = 0$) for some bounded sequence of constants $c_1^{(k)}>\frac{d_1-1}{T_1(\frac{1}{2})}$ and $c_2^{(k)}\in (\frac{\alpha}{\sup T_2},\frac{\alpha}{\inf T_2})$ and:
 \begin{itemize}
  \item $(y_2^{(k)})'(1)=0$ (smoothness condition for $f_2$),
  \item $(y_i^{(k)})'\le 0$,
  \item $(h^{(k)})'(\frac{1}{2})=(y_1^{(k)})'(\frac{1}{2})=(y_2^{(k)})'(\frac{1}{2})=0$ (evenness around $\frac{1}{2}$),
  \item $e^{-2y^{(k)}_1}(\frac{1}{2})=a$, $\displaystyle\sup_{t\in [\frac{1}{2},1]} e^{-2y^{(k)}_2(t)}+((y_2^{(k)})'(t))^2=S$,
  \item $(h^{(k)},y_1^{(k)},y_2^{(k)})$ is unbounded in $C^0[\frac{1}{2},1]\times C^1[\frac{1}{2},1]\times C^1[\frac{1}{2},1]$. 
 \end{itemize}
 Then there exists $(h,y_1,y_2)$ satisfying \eqref{NEF2S'} on $[\frac{1}{2},1)$, 
 such that $\left|y_2'\right|$ is bounded, and $y_i'\le 0$, and so that $(h,y_1,y_1',y_2,y_2')$ is unbounded about $1$. 
 Also, $e^{-2y_1}(\frac{1}{2})=a$, $\sup_{t\in [\frac{1}{2},1)} e^{-2y_2(t)}+(y_2'(t))^2=S$. 
\end{theorem}
By passing to a subsequence, we can assume that $c_1^{(k)}$, $c_2^{(k)}$ and $\gamma^{(k)}=e^{-2y_2^{(k)}(\frac{1}{2})}$ depend monotonically on $k$, 
so that all three sequences are convergent (we write $\hat{c}_1=\lim_{k\to \infty}c_1^{(k)}$). 
The proof of Theorem \ref{LOKS} is essentially Lemma \ref{LLTL}, which is possible to prove after all of the other lemmas in this section. 
\begin{lemma}\label{BATS}
Suppose $(h^{(k)},y_1^{(k)},y_2^{(k)})$ satisfies the hypothesis of Theorem \ref{LOKS}. 
Then for each sufficiently small $\epsilon>0$, $(h^{(k)},y_1^{(k)}, (y_2^{(k)})')$ 
is uniformly bounded in $C^0[\frac{1}{2},\frac{1}{2}+\epsilon]\times C^1[\frac{1}{2},\frac{1}{2}+\epsilon]\times C^0[\frac{1}{2},\frac{1}{2}+\epsilon]$.
\end{lemma}
\begin{proof}
 We can think of \eqref{NEF2S'} as a system of first order ODEs for $(h^{(k)},y_1^{(k)},(y_1^{(k)})',(y_2^{(k)})')$ 
 with initial conditions $(h^{(k)}(\frac{1}{2}),-\frac{\ln(a)}{2},0,0)$. 
 Since $e^{-2y_2^{(k)}}$ is uniformly bounded in $C^1[\frac{1}{2},1]$, standard ODE results would imply that we can find $\epsilon>0$ sufficiently small so that we obtain the desired bounds, provided $h^{(k)}(\frac{1}{2})$ is bounded. 

If we assume, for the sake of contradiction, that $\lim_{k\to \infty}h^{(k)}(\frac{1}{2})=\infty$, then by \eqref{ICH} we have that
\begin{align}\label{eqn:2limits}
 \lim_{k\to \infty}(d_1-1-c_1^{(k)}T_1(\tfrac{1}{2}))=\lim_{k\to \infty}\gamma^{(k)}(\alpha-c_2^{(k)}T_2(\tfrac{1}{2}))=0.
\end{align}

We now supress reference to the superscript $k$. 
We let $K=-\frac{d_1}{2}T_1'e^{2v_1}-\frac{d_2c_2}{2c_1}T_2'e^{2v_2}\ge 0$, where the $v_i$ functions are as defined in Section \ref{Rewriting}. Note that $K'\ge -\frac{d_1}{2}T_1''e^{2v_1}-\frac{d_2}{2}T_2''e^{2v_2}+2\min\{v_1',v_2'\}K$. 
Choose $t_1>\frac{1}{2}$ so that on $[\frac{1}{2},t_1]$, 
\begin{align}
 -\tfrac{d_1}{2}T_1''e^{2v_1}-\tfrac{d_2}{2}T_2''e^{2v_2}=e^{2v_1}\left(-\tfrac{d_1}{2}T_1''-\tfrac{d_2}{2}T_2''e^{2v_2-2v_1}\right)\ge e^{2v_1}(t_1-\tfrac{1}{2})
\end{align}
 for all $k$ 
(such a choice is possible 
because $T_1''(\frac{1}{2})<0$, $T_2''(\frac{1}{2})\le 0$, and $e^{2v_2-2v_1}=e^{2y_1-2y_2}\le \frac{\gamma e^{\sqrt{S}}}{a}$ 
by the conditions in the hypothesis of Theorem \ref{LOKS}).
 We claim that for $\delta=\frac{1}{5}$ (any $\delta<\frac{1}{4}$ will do), 
 \begin{equation}\label{ME1.5}
  v_i'\ge (1-2\delta) K
 \end{equation}
 on $[\frac{1}{2},t_1]$. 
 To see this, we use equation 
 \eqref{TMEFV} to find 
 \begin{align}
    v_i''&=(d_1-\delta_1^i)e^{2v_1}(d_1-1-c_1T_1)+(d_2-\delta_2^i)e^{2v_2}(\alpha-c_2T_2)+K(v_i'-K)+K' \nonumber\\
    &\ge e^{2v_1}\left((d_1-\delta_1^i)(d_1-1-c_1T_1(\tfrac{1}{2}))+(d_2-\delta_2^i)e^{2y_1-2y_2}(\alpha-c_2T_2(\tfrac{1}{2}))\right)+K(v_i'-K)+K' \nonumber\\
    & \ge e^{2v_1}\left((d_1-\delta_1^i)(d_1-1-c_1T_1(\tfrac{1}{2}))+(d_2-\delta_2^i)\tfrac{\gamma e^{\sqrt{S}}}{a}(\alpha-c_2T_2(\tfrac{1}{2}))+\delta (t_1-\tfrac{1}{2})\right) \nonumber\\ \label{eqnA}
    &\hspace{7cm}+2 \delta \min\{v_1',v_2'\}K+(v_i'-K)K+(1-\delta)K'
\end{align}
on $[\frac{1}{2},t_1]$. 
Using \eqref{eqn:2limits}, we see that the first term of \eqref{eqnA} is strictly positive for all $k$ large enough. 
If we assume \eqref{ME1.5} holds up until time $t_2\in [\frac{1}{2},t_1]$, we then have 
\begin{align*}
v_i''&> 2\delta(1-2\delta)K^2-2\delta K^2+(1-\delta)K'
=-4\delta^2 K^2+(1-\delta)K'
=(1-2\delta)K'+\delta(K'-4\delta K^2)\\
&\ge (1-2\delta)K'+\delta(2(1-2\delta)-4\delta)K^2\\
&\ge (1-2\delta)K'
\end{align*} 
on $[\frac{1}{2},t_2]$. 
Therefore, \eqref{ME1.5} is preserved on $[\frac{1}{2},t_1]$, and $t_1$ is independent of $k$. Clearly \eqref{ME1.5} holds at $t=\frac{1}{2}$, so it holds on $[\frac{1}{2},t_1]$. 
Now \eqref{ME1.5} can be thought of as a first order differential inequality for $v_1$. 
This inequality features a term $e^{v_1}$, so since
$v_1$ is defined on $[\frac{1}{2},t_1]$, we find that $v_1(\frac{1}{2})$ is bounded independently of $k$, which contradicts $\lim_{k\to \infty}h^{(k)}(\frac{1}{2})=\infty$.
\end{proof}

\begin{lemma}\label{UBI}
 Let $(h^{(k)},y_1^{(k)},y_2^{(k)})$ be as in the statement of Lemma \ref{BATS}. Then for each sufficiently small $\epsilon>0$, 
 $(h^{(k)},y_1^{(k)},(y_2^{(k)})')$ is uniformly bounded in $C^0[\frac{1}{2}+\epsilon,1-\epsilon]\times C^1[\frac{1}{2}+\epsilon,1-\epsilon]\times C^0[\frac{1}{2}+\epsilon,1-\epsilon]$.
\end{lemma}
\begin{proof}
 We already know that
 $(y_2^{(k)})'^2\le S$, and $y_1^{(k)}(\frac{1}{2})$ is fixed, so it suffices to show that 
 both $(y_1^{(k)})'$ and $h^{(k)}$ are bounded in $C^0[\frac{1}{2}+\epsilon,1-\epsilon]$. 
 
 We now assume that 
$(y_1^{(k)})'(t^{(k)})$ is unbounded for some sequence of times $t^{(k)}\in [\frac{1}{2}+\epsilon,1-\epsilon]$. 
First note that $(y^{(k)}_1)'(t)\le 0$ for $t\in [\frac{1}{2},1]$, so for $(y_1^{(k)})'(t^{(k)})$ to be unbounded, we must have $(y_1^{(k)})'(t^{(k)})\to -\infty$, after possibly taking a subsequence. From the equation for $(y^{(k)}_1)''$ in \eqref{NEF2S'}, we find 
\begin{align}\label{IEY1}
(y_1^{(k)})''&\le(h^{(k)})^2 e^{-2y_1^{(k)}}\left((d_1-1-c_1^{(k)}T_1)-\tfrac{d_1}{2}T_1'(y_1^{(k)})'\right).
\end{align}
Now since $T_1'<0$ on $[\frac{1}{2}+\epsilon,1-\frac{\epsilon}{4}]$ and is continuous, there is a $T^*<0$ so that $\frac{d_1}{2}T_1'\le T^*$ on 
$[\frac{1}{2}+\epsilon,1-\frac{\epsilon}{4}]$. Therefore, since $(y_1^{(k)})'$ gets arbitrarily large somewhere on $[\frac{1}{2}+\epsilon,1-\epsilon]$ and $c_1^{(k)}$ is bounded, we see from \eqref{IEY1} that 
\begin{align}
    \max_{t\in [1-\epsilon,1-\frac{\epsilon}{4}]}(y_1^{(k)})'(t)\to -\infty.
\end{align} 
Therefore, for each $t\in [1-\frac{\epsilon}{2},1]$, we have
\begin{align}
y_1^{(k)}(t)&\le y_1^{(k)}(1-\tfrac{\epsilon}{2})
=y_1^{(k)}(1-\epsilon)+\int_{1-\epsilon}^{1-\frac{\epsilon}{2}}(y_1^{(k)})' \nonumber \\
&\le -\frac{\ln(a)}{2}+\frac{\epsilon}{2}\max_{t\in [1-\epsilon,1-\frac{\epsilon}{2}]}(y_1^{(k)})'(t)\to -\infty.
\end{align}
Combining with the fact that $e^{-2y_2^{(k)}}\le S$, we find that 
 $\max_{t\in [1-\frac{\epsilon}{2},1]}e^{2y^{(k)}_1(t)-2y^{(k)}_2(t)}\to 0$. Then using the expression for $(h^{(k)})^2$ in
 \eqref{NEF2S'H2} and the fact that $(y_2^{(k)})'^2\le S$, we see that
\begin{align}
\frac{(h^{(k)})^2e^{-2y_1^{(k)}}}{(y_1^{(k)})'^2}&=\frac{d_1(d_1-1)+\frac{2d_1d_2(y_2^{(k)})'}{(y_1^{(k)})'}+\frac{d_2(d_2-1)(y_2^{(k)})'^2}{(y_1^{(k)})'^2}-\frac{c_1^{(k)}}{(y_1^{(k)})'^2}}{d_1(d_1-1-c_1^{(k)}T_1)+d_2 e^{2y_1^{(k)}-2y_2^{(k)}}(\alpha-c_2^{(k)} T_2)} \nonumber \\
&\to \frac{d_1-1}{d_1-1-\hat{c}_1T_1}
\end{align}
uniformly on $[1-\frac{\epsilon}{2},1-\frac{\epsilon}{4}]$. 
Now, re-writing the equation for $(y^{(k)}_1)''$ in \eqref{NEF2S'}, we get
\begin{align}\label{EFKY1}
(y_1^{(k)})''&=\frac{(h^{(k)})^2e^{-2y_1^{(k)}}}{(y_1^{(k)})'^2}\left((y_1^{(k)})'^2(d_1-1-c_1^{(k)}T_1)-(y_1^{(k)})'^3(\tfrac{d_1}{2}T_1'+\tfrac{d_2c_2^{(k)}}{2c^{(k)}_1}T_2'e^{2y^{(k)}_1-2y^{(k)}_2})\right).
\end{align}
Therefore, for large enough $k$, we have 
\begin{align}
(y_1^{(k)})''&\le 2\frac{d_1-1}{d_1-1-\hat{c}_1T_1}\left((y_1^{(k)})'^2(d_1-1-c_1^{(k)}T_1)-(y_1^{(k)})'^3T^*\right)
\end{align}
on $[1-\frac{\epsilon}{2},1-\frac{\epsilon}{4}]$, and $(y_1^{(k)})'(1-\frac{\epsilon}{2})$ going to $-\infty$. 
Such solutions cannot be made to exist on the entirety of $[1-\frac{\epsilon}{2},1-\frac{\epsilon}{4}]$ because of the dominance of the cubic term, so we obtain a contradiction. 
Therefore, $(y_1')^{(k)}$ is bounded in $C^0([\frac{1}{2}+\epsilon,1-\epsilon])$. 

Having already established bounds for $(y_1')^{(k)}$ and $(y_2')^{(k)}$, we can similarly show that
if $h^{(k)}$ gets too large on the interior, then $h^{(k)}$ blows up too early since there is a dominant cubic term (with the correct sign) in the third equation of \eqref{NEF2S'}. 
\end{proof}
\begin{lemma}\label{IEY10}
 Let $(h^{(k)},y_1^{(k)}, y_2^{(k)})$ be as in the statement of Lemma \ref{UBI}. Then \begin{align*}
                                                                                       \lim_{(k,t)\to (\infty,1)}e^{2y_1^{(k)}(t)-2y_2^{(k)}(t)}=0,
                                                                                      \end{align*}
 by which we mean that for each $\delta'>0$, there is a $K'>0$ and an $\epsilon'>0$ so that $e^{2y_1^{(k)}(t)-2y_2^{(k)}(t)}\le \delta'$ 
 whenever $k\ge K'$ and $t\in [1-\epsilon',1]$. 
\end{lemma}
\begin{proof}
Recall that, by definition, $\gamma^{(k)}=e^{-2y_2^{(k)}(\frac{1}{2})}$. Since $y_1^{(k)}$ is a decreasing function and $(y_2^{(k)})'\le \sqrt{S}$, we have that 
$e^{2y_1^{(k)}(t)-2y_2^{(k)}(t)}\le \frac{\gamma^{(k)} e^{\sqrt{S}}}{a}$ so the result is obvious if $\lim_{k\to \infty}\gamma^{(k)}=0$. 
On the other hand we also have 
\begin{align*}
 e^{2y_1^{(k)}(t)-2y_2^{(k)}(t)}\le e^{2y_1^{(k)}(t)}S,
\end{align*} 
so if $\lim_{k\to \infty}\gamma^{(k)}=\hat{\gamma}>0$, 
it suffices to show that $\lim_{(k,t)\to (\infty,1)}e^{2y_1^{(k)}(t)}=0$. 

Now we assume that $\lim_{k\to \infty}\gamma^{(k)}=\hat{\gamma}>0$. If, for the sake of contradiction, we assume that 
$\lim_{(k,t)\to (\infty,1)}e^{2y_1^{(k)}(t)}\neq 0$, then we can take a subsequence in $k$ if necessary, so that there is a sequence of times $t^{(k)}\to 1$ and an $a'>0$ so that $e^{2y_1^{(k)}(t^{(k)})}\ge a'$. Fix an $\epsilon>0$, and take a further subsequence if necessary so that $t^{(k)}\ge1- \epsilon$ for all $k$. Since $\lim_{k\to \infty}\gamma^{(k)}=\hat{\gamma}>0$, Lemmas \ref{BATS} and \ref{UBI} imply that the sequence $(h^{(k)}, y_1^{(k)}, y_2^{(k)})$ is uniformly bounded in $C^0[\frac{1}{2},1-\epsilon]\times C^1[\frac{1}{2},1-\epsilon]\times C^1[\frac{1}{2},1-\epsilon]$. We can therefore take a limit to find a solution $(h, y_1, y_2)$ of \eqref{NEF2S'} on $[\frac{1}{2},1-\epsilon]$. 
 Note that $e^{2y_1(t)}\ge a'$ for all $t\in [\frac{1}{2},1-\epsilon]$, since $e^{2y_1^{(k)}(t)}\ge e^{2y_1^{(k)}(t^{(k)})}\ge a'$ for each 
 $t\in [\frac{1}{2},1-\epsilon]$. 
 
 Now since the initial conditions for $(h,y_1,y_2)$ are actually independent of $\epsilon$, uniqueness of solutions to ODEs implies that the solution we get is independent of $\epsilon>0$, and can therefore be extended to a solution on $[\frac{1}{2},1)$, and we still have $e^{2y_1}\ge a'$ and $e^{-2y_2}+(y_2')^2\le S$. This implies that this solution does not satisfy the right smoothness conditions at $t=1$, 
 so the contrapositive of Theorem \ref{RicciRegularity} (with \eqref{WeakRegularity} replaced by \eqref{StrongRegularity}) then implies that $(h,y_1,y_2)$ must be bounded in $C^0[\frac{1}{2},1]\times C^1[\frac{1}{2},1]\times C^1[\frac{1}{2},1]$. However, 
 these bounds survive small perturbations in the initial conditions and small perturbations in $(c_1,c_2)$, 
 so we obtain a contradiction with the unboundedness of our original sequence $(h^{(k)},y_1^{(k)},y_2^{(k)})$. 
\end{proof}

\begin{lemma}\label{BY1D}
 Let $(h^{(k)},y_1^{(k)},y_2^{(k)})$ be as in the statement of Lemma \ref{IEY10}. Then for each $\delta>0$, 
 there exists an $\epsilon>0$ and a $K>0$ so that $-(y_1^{(k)})'(t)\le \frac{1+\delta}{1-t}$ whenever $k\ge K$ and $t\in [1-\epsilon,1]$. 
\end{lemma}
\begin{proof}
If it were not true, then we could find a $\delta>0$ and sequence of times $t^{(k)}\to 1$ so that 
$(y_1^{(k)})'(t^{(k)})\le -\frac{1+\delta}{(1-t^{(k)})}$, after possibly taking a subsequence. 
Then for large enough $k$, we claim that
\begin{align}\label{NEY1D}
 (y_1^{(k)})'(t)\le -\frac{1+\delta}{(1-t)}
\end{align}
for each $t\ge t^{(k)}$. 
To see this, we argue by contradiction and take a sequence of times $\tilde{t}^{(k)}\in [t^{(k)},1)$ so that \eqref{NEY1D} holds at $\tilde{t}^{(k)}$, but fails to hold immediately afterwards. Then $\lim_{k \to \infty} (y_1^{(k)})' (\tilde{t}^{(k)}) = - \infty$. Now letting $A^{(k)} =\frac{(h^{(k)})^2 e^{-2y_1^{(k)}}}{(y_1^{(k)})'^2}(\tilde{t}^{(k)})$, we see that
\begin{align}\label{LFA}
\lim_{k\to \infty}A^{(k)}= 1
\end{align}
by using \eqref{NEF2S'H2},
since $0\le e^{2y_1^{(k)}(\tilde{t}^{(k)})-2y_2^{(k)}(\tilde{t}^{(k)})} \to 0$ by Lemma \ref{IEY10}. Now we have
\begin{align}
(y_1^{(k)})''(\tilde{t}^{(k)})\le A^{(k)} \left((y_1^{(k)})'(\tilde{t}^{(k)})\right)^2\left(d_1-1-\frac{d_1(-T_1'(\tilde{t}^{(k)}))(1+\delta)}{2(1-\tilde{t}^{(k)})}\right)
\end{align}
for large $k$. Then since $\lim_{k\to \infty}\frac{-T_1'(\tilde{t}^{(k)})}{2(1-\tilde{t}^{(k)})}=1$ (by the smoothness conditions for $T$), we can use \eqref{NEY1D} (at $t=\tilde{t}^{(k)}$) and \eqref{LFA} and to find
\begin{align}
(y_1^{(k)})''(\tilde{t}^{(k)})<-\frac{1+\delta}{(1-\tilde{t}^{(k)})^2}
\end{align}
for large $k$, 
which contradicts the definition of $\tilde{t}^{(k)}$. 
Therefore, \eqref{NEY1D} holds for large $k$ and $t\ge t^{(k)}$, but this too is a contradiction because it implies that $y_1^{(k)}$ cannot exist 
on the entirety of $[t^{(k)},1]$. 
\end{proof}
\begin{lemma}\label{L2BH}
 Let $(h^{(k)},y_1^{(k)},y_2^{(k)})$ be as in the statement of Lemma \ref{BY1D}. Then $h^{(k)}$ is uniformly bounded in $L^4[\frac{1}{2},1]$. 
\end{lemma}
\begin{proof}
The idea of this proof is to carefully estimate $h^{(k)}$ towards the $t=1$ end of the interval by 
using the third equation of \eqref{NEF2S'}.

To start, choose a $\delta>0$ so that $d_1(1+\delta)(\delta^3+3\delta^2+3\delta)<\frac{1}{4}$. 
For ease of notation, we supress reference to the superscript $k$ in this proof. Since $T_1(1)=0$, $\alpha-c_2T_2(1)\ge 0$, $e^{-2y_1}\ge a$ and $e^{-2y_2}\le S$, we see that
\begin{align}\label{EFFTH'}
 \frac{d_1(d_1-1)}{d_1(d_1-1-c_1T_1) + (\alpha-c_2T_2) e^{2y_1-2y_2} } < 1 + \delta
\end{align}
on $[1-\epsilon,1]$ if $\epsilon$ is small enough, and $\epsilon$ can be made independent of $k$. Also, by Lemma \ref{IEY10} and the smoothness conditions for $T_1,T_2$, we can find a possibly smaller $\epsilon>0$ and $K>0$ so that whenever $k\ge K$, we have 
 \begin{align}\label{BFET}
  \left(-\frac{d_1T_1'}{2}-\frac{d_2c_2}{2c_1}T_2'e^{2y_1-2y_2}\right)\le d_1(1+\delta)(1-t)
 \end{align}
 on $[1-\epsilon,1]$ for all $k\ge K$. Finally, by Lemma \ref{BY1D}, we can make $\epsilon>0$ smaller and $K>0$ larger if necessary so that for $k\ge K$ and 
 $t\in [1-\epsilon,1]$, 
 \begin{align}\label{ICOL21}
  (1-t)\left|y_1'\right|\le 1+\delta,
 \end{align}
 so in particular, there exists an $N>0$ depending on $\delta$ so that 
 \begin{align}\label{DefNH}
  \frac{(d_2^2S+2d_1d_2\sqrt{S}\left|y_1'\right|)d_1(1+\delta)(1-t)}{d_1(d_1-1-c_1T_1)+d_2 e^{2y_1-2y_2}(\alpha-c_2 T_2)}\le N,
 \end{align} 
 whenever $k\ge K$ and $t\ge 1-\epsilon$. 
 
Now one of the terms in the third equation of \eqref{NEF2S'} is  
$-\frac{h^2}{2}\left(d_1T_1'e^{-2y_1}+\frac{d_2c_2}{c_1}T_2'e^{-2y_2}\right)$.
By \eqref{BFET} we can estimate:
\begin{align}\label{FEFTH'}
-\frac{h^2}{2}\left(d_1T_1'e^{-2y_1}+\frac{d_2c_2}{c_1}T_2'e^{-2y_2}\right)&\le d_1(1+\delta)(1-t)h^2e^{-2y_1}
\end{align}
on $[1-\epsilon,1]$, and for each $k\ge K$. Using \eqref{NEF2S'H2} we can also estimate the $h^2$ term:
\begin{align}
 h^2&\le \frac{d_1^2y_1'^2+d_2^2y_2'^2+2d_1d_2y_1'y_2'-d_1y_1'^2-d_2y_2'^2}{d_1 (d_1-1-c_1T_1) e^{-2y_1} + d_2 (\alpha-c_2 T_2) e^{-2y_2}} \nonumber \\
 \label{eqn:h2le}
 &\le \frac{d_1(d_1-1)y_1'^2 + d_2^2 S +2d_1d_2\sqrt{S}\left|y_1'\right|}{d_1 (d_1-1-c_1T_1) e^{-2y_1} + d_2 (\alpha-c_2 T_2) e^{-2y_2}}.
\end{align}

Using \eqref{FEFTH'}, \eqref{eqn:h2le}, \eqref{EFFTH'} and \eqref{DefNH}, we then find 
\begin{align}
&-\frac{h^2}{2}\left(d_1T_1'e^{-2y_1}+\frac{d_2c_2}{c_1}T_2'e^{-2y_2}\right) \nonumber \\
&\le  d_1(1+\delta)(1-t)\left(\frac{d_1(d_1-1)y_1'^2 + d_2^2 S + 2d_1d_2\sqrt{S}\left|y_1'\right|}{d_1(d_1-1-c_1T_1)+d_2 (\alpha-c_2 T_2) e^{2y_1-2y_2}}\right)\\
&=\frac{d_1(1+\delta)(1-t)d_1(d_1-1)y_1'^2}{d_1(d_1-1-c_1T_1)+d_2 (\alpha-c_2 T_2) e^{2y_1-2y_2}} + \frac{(d_2^2S+2d_1d_2\sqrt{S}\left|y_1'\right|)d_1(1+\delta)(1-t)}{d_1(d_1-1-c_1T_1)+d_2 (\alpha-c_2 T_2) e^{2y_1-2y_2}}\\
&\le d_1(1+\delta)^2(1-t)y_1'^2+N.
\end{align}
Then by using \eqref{ICOL21} again, we find
\begin{align}
-\frac{h^2}{2}\left(d_1T_1'e^{-2y_1}+\frac{d_2c_2}{c_1}T_2'e^{-2y_2}\right)\le d_1(1+\delta)^3 (-y_1')+N.
\end{align}
We now reinstate the superscript $k$, and turn to the third equation of \eqref{NEF2S'}, finding
\begin{equation}\label{EFHINOFY1}
\begin{aligned}
\frac{(h^{(k)})'}{h^{(k)}}&\le d_1(y_1^{(k)})^{'}+d_1(1+\delta)^3 \left(-(y_1^{(k)})'\right) + N\\
&= d_1(\delta^3+3\delta^2+3\delta)\left(-(y_1^{(k)})'\right)+N\le \frac{d_1(1+\delta)(\delta^3+3\delta^2+3\delta)}{(1-t)}+N,
\end{aligned}
\end{equation}
for $t\in [1-\epsilon,1]$. Integrating the inequality \eqref{EFHINOFY1}, we arrive at
$\frac{h^{(k)}(t)}{h^{(k)}(1-\epsilon)}\le \frac{e^{N(t-(1-\epsilon))}\epsilon^D}{(1-t)^{D}}$, where $D=d_1 (1+\delta)(\delta^3+3\delta^2+3\delta)$. By our choice of $\delta$, $D < \frac{1}{4}$, so $\frac{(h^{(k)})}{h^{(k)}(1-\epsilon)}$ is bounded uniformly in $L^4[1-\epsilon,1]$, hence 
$h^{(k)}$ is bounded uniformly in $L^4[\frac{1}{2},1]$ because we have already established uniform bounds for $h^{(k)}$ in $C^0[\frac{1}{2},1-\epsilon]$. 
\end{proof}

\begin{lemma}\label{GDNGT0}
 Let $(h^{(k)},y_1^{(k)},y_2^{(k)})$ be as in the statement of Lemma \ref{L2BH}. Then $y_2^{(k)}$ is uniformly bounded in $C^1[\frac{1}{2},1]$. 
\end{lemma}
\begin{proof}
Since $(y_2^{(k)})'$ and $e^{-2y_2^{(k)}}$ are both uniformly bounded on $[\frac{1}{2},1]$, it suffices to show that 
$\gamma^{(k)}=e^{-2y_2^{(k)}(\frac{1}{2})}$ does not converge to $0$. 
If $\lim_{k\to \infty}\gamma^{(k)}=0$, we would find that $e^{-2y_2^{(k)}}$ converges to $0$ uniformly on $[\frac{1}{2},1]$. 
The equation for $(y_2^{(k)})''$ in \eqref{NEF2S'} and the fact that $(y_2^{(k)})'$ and 
$\left(\frac{d_1}{2}T_1'e^{-2y_1}+\frac{d_2c_2 }{2c_1} T_2' e^{-2y_2}\right)$ are both non-positive implies that
\begin{align}
(y_2^{(k)})''\le (h^{(k)})^2 (\alpha-c_2^{(k)}T_2) e^{-2y_2^{(k)}}. 
\end{align}
Now, we know that $(y_2^{(k)})'$ is negative, $e^{-2y_2^{(k)}}$ converges to $0$ uniformly 
and $\sup_{[\frac{1}{2},1]}(e^{-2y_2^{(k)}}+((y^{(k)}_2)')^2)-S=0$. Therefore we can find a sequence of times $t^{(k)}$ so that $(y_2^{(k)})'(t^{(k)})\le-\frac{\sqrt{S}}{2}$. Then
\begin{align*}
 0&=(y_2^{(k)})'(1)\\
 &=(y_2^{(k)})'(t^{(k)})+\int_{t^{(k)}}^{1}(y_2^{(k)})''(s)ds\\
 &\le -\frac{\sqrt{S}}{2}+\left|\left|(h^{(k)})^2(\alpha-c_2^{(k)}T_2) e^{-2y_2^{(k)}}\right|\right|_{L^1[\frac{1}{2},1]}.
\end{align*}
Notice that $(\alpha-c_2^{(k)}T_2)$ is bounded, $e^{-2y_2^{(k)}}$ converges to $0$ uniformly, and 
$(h^{(k)})^2$ is bounded in $L^1[\frac{1}{2},1]$ (by Lemma \ref{L2BH}), so $\left|\left|(h^{(k)})^2e^{-2y_2^{(k)}}(\alpha-c_2^{(k)}T_2)\right|\right|_{L^1[\frac{1}{2},1]}\to 0$, a contradiction. 

\end{proof}

Finally, we have enough information to take limits. 

\begin{lemma}\label{LLTL}
 Let $(h^{(k)}, y_1^{(k)}, y_2^{(k)})$ be as in the statement of Lemma \ref{GDNGT0}. Then there exists $(h,y_1,y_2)$ satisfying \eqref{NEF2S'} on $[\frac{1}{2},1)$, 
 such that $\left|y_2'\right|$ is bounded, and $y_i'\le 0$, and so that $(h, y_1, y_1', y_2, y_2')$ is unbounded about $1$. Also, $e^{-2y_1}(\frac{1}{2})=a$, $\sup_{t\in [\frac{1}{2},1]} e^{-2y_2(t)}+(y_2'(t))^2=S$. 
\end{lemma}
\begin{proof}
We now know that for each $\epsilon>0$, $(h^{(k)},y_1^{(k)},y_2^{(k)})$ is bounded in 
$C^0[\frac{1}{2},1-\epsilon]\times C^1[\frac{1}{2},1-\epsilon]\times C^1[\frac{1}{2},1-\epsilon]$. The fact that 
these functions solve \eqref{NEF2S'} implies that we actually have uniform bounds in
$C^1[\frac{1}{2},1-\epsilon]\times C^2[\frac{1}{2},1-\epsilon]\times C^2[\frac{1}{2},1-\epsilon]$, whence we get sub-convergence in 
$C^0[\frac{1}{2},1-\epsilon]\times C^1[\frac{1}{2},1-\epsilon]\times C^1[\frac{1}{2},1-\epsilon]$ to a limit $(h,y_1,y_2)$ which still solves \eqref{NEF2S'}. 
It is clear that $h>0$ on $[\frac{1}{2},1-\epsilon]$. 

Now for each $\epsilon>0$, the limiting solution solves the same ODEs with the same initial conditions. Therefore, the solution is actually independent of $\epsilon>0$, 
so by sending $\epsilon$ to $0$, we get a solution defined on $[\frac{1}{2},1)$. Note that on $[\frac{1}{2},1)$, we have 
\begin{align}
 e^{-2y_2}\le S, \qquad e^{2y_1}\le \frac{1}{a}, \\
 y_1'\le 0, \qquad -\sqrt{S}\le y_2'\le 0,
\end{align}
because the corresponding statements are true for the sequence.

Now this solution cannot have $h, y_1, y_1', y_2, y_2'$ bounded on the entirety of $[\frac{1}{2},1)$, otherwise they could be extended to a solution on all of $[\frac{1}{2},1]$. Then by perturbing 
$\gamma$ (which will perturb $y_2(\frac{1}{2})$ and $h(\frac{1}{2})$), 
as well as $c_1$ and $c_2$, we see that the solution is bounded for all nearby $\gamma$ and $(c_1,c_2)$, which contradicts the unboundedness of our original sequence. 

Finally, we also show that $\sup_{t\in [\frac{1}{2},1)} e^{-2y_2(t)}+(y_2'(t))^2=S$. It is clear that $e^{-2y_2(t)}+(y_2'(t))^2\le S$ for each $t\in [\frac{1}{2},1)$ because the same is true for the sequence. Take a subsequence so that 
$e^{-2y_2^{(k)}(1)}$ is convergent. Now if $\lim_{k\to \infty}e^{-2y_2^{(k)}(1)}<S$, then since $(y_2^{(k)})' \le 0$, we can find an $\epsilon>0$ so that $e^{-2y_2^{(k)}}\le S-\epsilon$ on $[\frac{1}{2},1]$, for large enough $k$. Then there exists a sequence of times $t^{(k)}$ so that $((y_2^{(k)})' (t^{(k)}))^2 = S - e^{-2y_2^{(k)}(t^{(k)})}\ge \epsilon$, and we can again assume that $t^{(k)}$ is convergent to $t^*\in [\frac{1}{2},1]$. If $t^*<1$, then uniform convergence away from $t=1$ implies that $e^{-2y_2(t^*)}+(y_2'(t^*))^2=S$ as required. If $t^*=1$, then the fact that $(y_2^{(k)})'(t^{(k)})\le -\sqrt{\epsilon}$ is a contradiction. Indeed, the equation for $(y_2^{(k)})''$ in \eqref{NEF2S'} 
implies that on $[t^{(k)},1]$, we have 
\begin{align*}
 (y_2^{(k)})''\le (h^{(k)})^2(\alpha-c_2^{(k)}T_2)e^{-2y_2^{(k)}}\le (h^{(k)})^2\alpha S
\end{align*}
so since $\left|\left|(h^{(k)})^2\right|\right|_{L^1[t^{(k)},1]}\to 0$ (H\"older's inequality, coupled with the uniform bounds on $h^{(k)}$ in
$L^4[\frac{1}{2},1]$), we cannot possibly have $(y_2^{(k)})'(t^{(k)})\le -\sqrt{\epsilon}$ and 
$(y_2^{(k)})'(1)=0$ simultaneously. 

On the other hand, if $\lim_{k\to \infty}e^{-2y_2^{(k)}(1)}=S$, then we can use the estimate  $((y_2^{(k)})')^2 \le S$ to conclude that $\lim_{t\to 1}e^{-2y_2(t)}=S$. 
\end{proof}
%
\begin{proof}[Proof of Theorem \ref{RicciExistence}]
By Proposition \ref{EOIE} and the discussion following it, it is enough to obtain a solution to the equation \eqref{NEF2S'} (with $\beta = 0$). 
By Theorems \ref{TESS} and \ref{LOKS} we obtain a solution to \eqref{NEF2S'} on $[\frac{1}{2}, 1)$ 
that satisfies the right smoothness conditions at $t = \frac{1}{2}$, and is unbounded at $t = 1$. 
Then Theorem \ref{RicciRegularity} (with \eqref{WeakRegularity} replaced by \eqref{StrongRegularity}) implies that the corresponding metric $\g$ extends to a smooth metric on all of $M$.
\end{proof}


\appendix 
\section{Some Properties of Singular ODEs}
In this appendix, we state some facts about solutions to singular ODEs that are used in Section \ref{PRR}. 
\begin{prop}\label{TCsF}
 Let $x:(0,\epsilon]\to \mathbb{R}$ be a continuously differentiable function satisfying 
 \begin{align*}
  x'(t)=-\frac{a(t)(x(t)-c)}{t}+b(t),
 \end{align*}
 with $c\in \mathbb{R}$, and $a,b$ continuous functions on $[0,\epsilon]$, with $a(0)>0$. 
 If $x$ is continuous on $[0,\epsilon]$ with $x(0)=c$, then $x\in C^1[0,\epsilon]$, and 
 $x'(0)=\frac{b(0)}{1+a(0)}$. 
\end{prop}
\begin{proof}
 By adding an affine function to $x$, we can assume without loss of generality that $c=0$ and $b(0)=0$. 
 
 For each small $\tilde{\epsilon}>0$, we can find a $\delta>0$ so that for each $t\in [0,\delta]$, we have 
 \begin{align*}
  \left|b(t)\right|&<\tilde{\epsilon}\\
  \left|a(t)-a(0)\right|&<\tilde{\epsilon}\\
  a(t)&>0. 
 \end{align*}
 On $[0,\delta]$, consider the functions 
 \begin{align*}
  x_+(t)&=\frac{\tilde{\epsilon}}{1+a(0)-\tilde{\epsilon}}t\\
  x_-(t)&=\frac{-\tilde{\epsilon}}{1+a(0)-\tilde{\epsilon}}t\\
 \end{align*}
Then 
\begin{align*}
 x_+'(t)=\frac{\tilde{\epsilon}}{1+a(0)-\tilde{\epsilon}}\ge -\frac{a(t)x_+(t)}{t}+b(t),\\
 x_-'(t)=\frac{-\tilde{\epsilon}}{1+a(0)-\tilde{\epsilon}}\le -\frac{a(t)x_-(t)}{t}+b(t),
\end{align*}
and $x_+(0)=x_-(0)=0$. 
Therefore, $x(t)\in [x_-(t),x_+(t)]$, so 
$\frac{x(t)}{t}\in [\frac{-\tilde{\epsilon}}{1+a(0)-\tilde{\epsilon}},\frac{\tilde{\epsilon}}{1+a(0)-\tilde{\epsilon}}]$. 
This shows that $x'(0)$ exists, and is equal to $0$. 
In fact, since $x'(t)=-\frac{a(t)x(t)}{t}+b(t)$, and $a(t)$, $b(t)$ are both continuous, this also shows that 
$\lim_{t\to 0}x'(t)=0$, so $x\in C^1[0,\epsilon]$. 
\end{proof}
A simple time reversal of the last proposition gives

\begin{prop}\label{TCs}
 Let $x:[1-\epsilon,1)\to \mathbb{R}$ be a continuously differentiable function satisfying 
 \begin{align*}
  x'(t)=\frac{a(t)(x(t)-c)}{(1-t)}+b(t),
 \end{align*}
 with $c\in \mathbb{R}$, and $a,b$ continuous functions on $[1-\epsilon,1]$, with $a(1)>0$. 
 If $x$ is continuous on $[1-\epsilon,1]$ with $x(1)=c$, then $x\in C^1[1-\epsilon,1]$, and 
 $x'(1)=\frac{b(1)}{1+a(1)}$. 
\end{prop}
\section{The Leray-Schauder Degree}
Throughout Section \ref{AASOS}, we use the notion of the Leray-Schauder degree. 
In this appendix we will describe the relevant results about the Leray-Schauder degree, as taken from pages 185-225 of \cite{LSDegree}.

Suppose $X$ is a Banach space, and $\Omega\subset X$ is bounded, open and convex. If $I:X\to X$ is the identity mapping and 
$f:\overline{\Omega}\to X$ is a compact mapping so that $x-f(x)\neq 0$ for each 
$x\in \partial \Omega$, then we can associate to $I-f$ the \textit{Leray-Schauder Degree}, denoted $\deg(I-f,\Omega,0)$, which provides a weighted count 
of the number of zeroes that $I-f$ has inside $\Omega$. The Leray-Schauder degree has useful properties, as described in the following. 
\begin{theorem}\label{PLSD}The Leray-Schauder degree has the following properties:\\
 (i) If $\deg(I-f,\Omega,0)\neq 0$, then there is an $x\in \Omega$ so that $x=f(x)$.\\
 (ii) If $H:[0,1]\times \Omega \to X$ is continuous, compact and $x-H(t,x)\neq 0$ for each $(t,x)\in [0,1]\times \partial \Omega$, then 
 $\deg(I-H(t,\cdot),\Omega,0)$ is independent of $t$. \\
 (iii) If $I-f'(a)$ is invertible for each fixed point $a$ of $f$, then $\deg(I-f,\Omega,0)=\sum_{a\in (I-f)^{-1}(0)}(-1)^{\sigma(a)}$, where 
 $\sigma(a)$ is the sum of of the algebraic multiplicities of the eigenvalues of $f'(a)$ contained in $(1,\infty)$. 
\end{theorem}
In this paper, we usually take $X=C^1([\frac{1}{2},1]:\mathbb{R}^2)\times \mathbb{R}\times \mathbb{R}$, and 
$\Omega$ to be $V\times (\frac{\alpha}{\sup T_2},\frac{\alpha}{\inf T_2})\times (\overline{\gamma},S+1)$, where $V$ is some open ball in $C^1([\frac{1}{2},1]:\mathbb{R}^2)$, and the other terms are described in the main text.
\section{Regularity}\label{DKRE}
In this appendix, we explain why, in the cohomogeneity one setting, it suffices to know that our metric is $C^1$ when dealing with regularity of solutions to the prescribed Ricci curvature equation. This result is a slight modification of Theorem 4.5 (a) of \cite{DeTurckKazdan}.
\begin{theorem}\label{DTG}
 Suppose $T$ is a smooth, non-degenerate symmetric $(0,2)$ tensor field on $M$, and there is a Riemannian metric $\g$ so that $\g\in C^1(M)$, and 
 $\g\in C^{\infty}(D)$ for some open and dense subset $D$ of $M$ so that $M\setminus D$ consists of finitely many, non-intersecting smooth and closed submanifolds 
 of $M$. If $\text{Ric}(\g)=T$ on $D$, then $\g\in C^{\infty}(M)$. 
\end{theorem}
\begin{proof}
Note that since $T$ is invertible, on $D$ we have 
 \begin{align}
  \text{Ric}(\g)-T+\text{div}^*\left(T^{-1}\text{Bian}(\g,T)\right)=0,
 \end{align}
 where $\text{div}$ is the divergence operator, and $\text{Bian}$ is the Bianchi tensor; the Bianchi identity tells us that $\text{Bian}(\g,T)=0$ since $\text{Ric}(\g)=T$. 
 Now for any point $p\in M$, we can find local coordinates on a ball $B$, so that $M\setminus D$ lies in the standard hypersurface, and 
 the equation can be written 
 \begin{align}\label{EQFRC}
  \g^{rs}\frac{\partial^2 \g_{ij}}{\partial x_r\partial x_s}=H_{ij}(\g,T), \ \text{for each} \ i,j,
 \end{align}
 where $H_{ij}$ is an analytic function in the components of $T$ and their derivatives, as well as 
 the first derivatives of $\g$ and its inverse components. We aim to show that $\g$ is in $C^2$ for some possibly smaller neighborhood of $p$; since $p$ 
 is arbitrary, this will imply that $\g$ is $C^2$ everywhere, so Theorem 4.5 (a) of \cite{DeTurckKazdan}
can be used to imply that $\g$ is smooth everywhere. 
 
 Note that for each $i,j$, \eqref{EQFRC} can be seen as a strongly elliptic scalar equation
 since the $\g^{rs}$ terms are positive-definite and lie in $C^1$. Choose a smooth cut-off function $\phi$ (it is 
 $1$ on some neighbourhood of $p$, and $0$ in some neighbourhood of $\partial B$) and define $u=\phi \g_{ij}$.   
 Note that $u\in C^1(\overline{B})$, $u$ is smooth on the upper-half ball $B_+$ and the lower-half ball $B_-$, and it also 
 solves the following equation:
 \begin{align}\label{EQFRC'}
 \begin{split}
  L(u)&=f \ \text{on} \ B_+\cup B_-,\\
  u&=0 \ \text{on} \ \partial B, 
  \end{split}
 \end{align}
 where $L(u)=\frac{\partial \left(a_{rs}\frac{\partial u}{\partial x_r}\right)}{\partial x_s}$, $a_{rs}=g^{rs}\in C^1(\overline{B})$ are uniformly elliptic coefficients, and 
 $f$ is a continuous function on $\overline{B}$. Our aim is to show that $u$ is actually a $W^{1,2}_0(B)$ function which weakly solves \eqref{EQFRC'}.
 To this end, fix an arbitrary $\psi\in C_c^{\infty}(B)$, and take smooth domains $D_n$ that consist of two disconnected components, one of which lies in and exhausts $B_+$, and the other lies in and exhausts $B_-$. Now on $\overline{D_n}$, $u$ is smooth and $L(u)=f$ in the classical sense, so 
 \begin{align}
  \int_{D_n}a_{rs}\frac{\partial u}{\partial x_r}\frac{\partial \psi}{\partial x_s}+\int_{D_n}f\psi=\int_{\partial D_n}a_{rs}\frac{\partial u}{\partial x_r} \psi\nu_s.
 \end{align}
 As $n\to \infty$, the right-hand side vanishes (since $u,\psi\in C^1(\overline{B})$, $\psi=0$ close to $\partial B$, and the hypersurface terms cancel out), so we obtain 
\begin{align}\label{USPW}
 \int_B a_{rs}\frac{\partial u}{\partial x_r}\frac{\partial \psi}{\partial x_s}=-\int_B f\psi
\end{align}
for each $\psi\in C_c^{\infty}(B)$. It then follows that \eqref{USPW} also holds for $\psi\in W^{1,2}_0(B)$, so $u$ is in fact the unique weak solution of 
\eqref{EQFRC'}. Theorem 8.34 of \cite{GT} then implies that $u\in C^{1,\alpha}(\overline{B})$ for each $\alpha\in (0,1)$. 

Then $\g$ is $C^{1,\alpha}$ in a smaller neighborhood $\tilde{B}$ of $p$, so the right-hand side of \eqref{EQFRC} lies in $C^{0,\alpha}(\overline{\tilde{B}})$.
Then for a new cut-off function $\tilde{\phi}$ in $\tilde{B}$, we can proceed as before to find that $\tilde{u}=\tilde{\phi}g_{ij}$ is a weak solution to 
$L(\tilde{u})=\tilde{f}$ with $\tilde{f}\in C^{0,\alpha}(\overline{\tilde{B}})$, so $\tilde{u}\in C^{2,\alpha}(\overline{\tilde{B}})$ by Theorem 6.14 of \cite{GT}.
This implies that $\g$ is $C^2$ in a neighborhood of $p$. 
\end{proof}

In this paper, the idea is to use Theorem \ref{DTG} on manifolds $M$ equipped with a cohomogeneity one group action $\G$, and $D$ is chosen to be the principal 
part of the group action. 
Since our metrics are $\G$-invariant, and chosen to satisfy certain ODEs,
it is clear that the metrics will be smooth on the principal orbits, so it becomes important to know when the metric will be $C^1$ on all of the manifold, including the singular 
orbits; this is discussed in Section \ref{Prelim}. 



\end{document}